\documentclass[final,1p,times]{elsarticle}

\usepackage{amsmath}
\usepackage{amsthm}
\usepackage{amssymb}
\usepackage{amsfonts}

\newtheorem{theorem}{Theorem}
\newtheorem{lemma}[theorem]{Lemma}

\newtheorem{definition}[theorem]{Definition}

\newtheorem{example}[theorem]{Example}
\newtheorem{remark}[theorem]{Remark}
\newtheorem{problem}[theorem]{Problem}
\newtheorem{acknowledgment}[]{Acknowledgment}

\allowdisplaybreaks

\newcommand{\RR}{\mathbb{R}}
\newcommand{\NN}{\mathbb{N}}
\newcommand{\ZZ}{\mathbb{Z}}   

\newcommand{\sS}{\mathcal{S}}

\newcommand{\set}[2]{\left\{#1 \;\middle\vert\; #2\right\}}

\newcommand{\witi}[1]{\widetilde{#1}}

\DeclareMathOperator{\D}{d}
\DeclareMathOperator{\rnk}{rank}
\DeclareMathOperator{\Hess}{Hess}
\DeclareMathOperator{\SA}{SA}

\DeclareMathOperator{\dest}{dest}

\usepackage{graphicx}     
\usepackage{subfig}       
\graphicspath{{Figures/}} 

\newenvironment{glists}[4]{
                 \begin{list}{}{
                     \setlength{\labelwidth}{#2}
                     \setlength{\labelsep}{#3}
                     \setlength{\leftmargin}{#1}
                     \addtolength{\leftmargin}{\labelwidth}
                     \addtolength{\leftmargin}{\labelsep}
                     \setlength{\parsep}{#4}
                     \setlength{\topsep}{\parsep}
                     \setlength{\itemsep}{\parsep}
                     \setlength{\listparindent}{0in}
                     }
                 }{
                 \end{list}
                 }
\newcommand{\iteml}[1]{\item[#1 \hfill]}
\newcommand{\itemr}[1]{\item[\hfill #1]}

\usepackage[linesnumberedhidden, boxed, vlined]{algorithm2e}

\usepackage{xcolor}
\definecolor{monvert}{rgb}{0.05,0.54,0.05}

\def\new#1{{\textcolor{black}{#1}}}

\begin{document}



\SetKwInOut{Input}{Input}
\SetKwInOut{Output}{Output}

\SetKwFor{Loop}{loop}{}{}

\SetKw{LOOP}{loop}
\SetKw{ELSE}{else}

\newcommand{\algrule}[1][.2pt]{\par\vskip.5\baselineskip\hrule height #1\par\vskip.5\baselineskip}


\begin{frontmatter}
\title{Connectivity in Semi-Algebraic Sets I}

\author{Hoon Hong}
\address{North Carolina State University}
\address{Raleigh, NC 27695, USA}
\ead{hong@ncsu.edu}

\author{James Rohal}
\address{U.S. Naval Research Laboratory}
\address{SW Washington, DC 20375, USA}
\ead{jjrohal@gmail.com}

\author{Mohab {Safey~El~Din}}
\address{Sorbonne Universit\'e, CNRS, LIP6}
\address{75005 Paris, France}
\ead{Mohab.Safey@lip6.fr}

\author{\'Eric Schost}
\address{University of Waterloo}
\address{Waterloo, ON   N2L 3G1, Canada}
\ead{eschost@uwaterloo.ca}

\begin{abstract}
   A semi-algebraic set is a subset of the real space defined by polynomial equations and inequalities having real coefficients and is a union of finitely many maximally connected components. We consider the problem of deciding whether two given points in a semi-algebraic set are connected; that is, whether the two points lie in the same connected component. In particular, we consider the semi-algebraic set defined by $f \neq 0$ where $f$ is a given polynomial with integer coefficients. The motivation comes from the observation that many important or non-trivial problems in science and engineering can be often reduced to that of connectivity. Due to its importance, there has been intense research effort on the problem. We will describe a symbolic-numeric method based on gradient ascent. The method will be described in two papers. The first  paper (the present one) will describe  the symbolic part and the forthcoming second paper will describe the numeric part. In the present paper, we give proofs of correctness and termination for the symbolic part and illustrate the efficacy of the method using several non-trivial examples.
\end{abstract}


\end{frontmatter}


\section{Introduction}
Many important or non-trivial problems in science and engineering can be reduced to that of ``connectivity;'' that is, deciding whether two given points in a given set can be connected via a continuous path within the set. Equivalently, it is a problem of deciding whether the two points lie in a same connected component of a given semi-algebraic set. 

In a series of papers \cite{
SchwartzSharir1983,
Schwartz1983,
SchwartzSharir1983a,
SharirAriel-Sheffi1984,
SchwartzSharir1984,
Schwartz1987},
Schwartz and collaborators developed the first rigorous methods 
based on  Collins' Cylindrical Algebraic Decomposition~\cite{Collins:75} 
and adjacency determination~\cite{
Arnon_Collins_McCallum:84b,
Arnon_Collins_McCallum:88,
Arnon:88}, 
which is based on repeated univariate resultants, 
with  doubly exponential complexity in the number of variables.
In \cite{Canny1988, Canny:93b}, 
Canny presented a method that explicitly builds a roadmap 
by using multivariate (Macaulay) resultant, 
with a single exponential complexity  
$s^{n}\log( s) d^{O\left(  n^{4}\right)}$ (deterministic) and 
$s^{n}\log^2 (s) d^{O\left(  n^{2}\right)}$ (randomized) 
where $n$ is \#\ of variables, 
$s$ is \# polynomials and 
$d$ is the maximum degree. 
It inspired intensive effort to improve the exponent: to name a few,
Gournay, Risler, Heintz, Roy, Solerno, Basu, Pollack, Roy, Grigoriev, 
Vorobjov, Safey El Din and Schost
\cite{
Heintz_Roy_Solerno:90b, 
Heintz_Roy_Solerno:90c,
GHRSV90, 
CGV92, 
GV92,
Gournay_Risler:93,
Heintz_Roy_Solerno:94a,
BaPoRo96,
BPRRoadmap,
SafeyElDin2010,
Basu2014b,
Basu2014,
SafeyEldin_Schost:2017}. 
The current state of the art is as follows: 
1. Deterministic\;~\cite{Basu2014}\;:   
polynomial in $n^{n\log^{3}(n)}d^{n\log^{2}(n)}$
(for arbitrary real algebraic set).
2. Probabilistic\;~\cite{SafeyEldin_Schost:2017}\;:    
polynomial in $\left( nd\right)^{n\log (k)}$ and sub quadratic in the
output size 
(for smooth-compact real algebraic set) which is near-optimal, where $k$ is the dimension of the real algebraic set.

Summarizing, through the intensive efforts during last several decades,
tremendous progress has been made, resulting in  asymptotically fast
(near-optimal) algorithms. Some attempts to make roadmap algorithms practical 
have been successful on some real-life problems coming from robotics (see e.g.
\cite{CSS20} for the analysis of kinematic singularities of some industrial
robots) but note that the behaviour of such algorithms depends on the
\emph{degree} of the (Zariski closure of the) roadmap itself and of some
algebraic sets needed to be  considered to compute them. Such degrees can
sometimes become an obstruction to tackle applications.

In this paper, we present an alternative approach with the hope of reducing the constant.
We will consider a crucial special case where the given set is a particular type of semi-algebraic set, in that it consists of the points where a given polynomial~$f$ is not equal to~$0$. We state the problem more precisely. 
\begin{problem}\label{problem}~\\
\begin{glists}{1em}{5em}{0em}{0em}
  \iteml{\textbf{Input}:}
  $f \in \mathbb{Z}[x_1,\ldots,x_n]$, $n \geq 2$, $\deg f \geq 1$, squarefree, with finitely many singular points, \\
  $p,q \in \mathbb{Q}^n \cap \{f \neq 0\}$ where
  \[
    \{f \neq 0\} := \set{x\in \mathbb{R}^n}{f(x) \neq 0}.
  \]
  \iteml{\textbf{Output}:}
  {\tt true}, if the two points $p$ and $q$
  lie in a same semi-algebraically connected component of 
  the set $\{f \neq 0\}$, else {\tt false}.
\end{glists}
\end{problem}
\medskip 
 
\begin{example}\label{ex:toy} We illustrate the problem using a toy example. Let
\begin{equation}\label{eq:f}
f = -2 x_1^2 + x_1^4 - 2 x_2^2 + 2 x_1^2 x_2^2 + x_2^4.
\end{equation}
Figure~\eqref{fig:fplot} below shows the curve defined by $f=0$. 
Note that its complement, $\{f \neq 0\}$, consists of two semi-algebraically connected components. 
\begin{figure}[ht]
\centering
\subfloat[]{\includegraphics{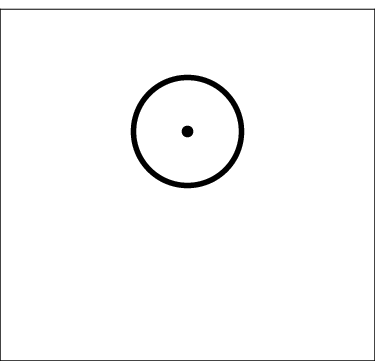}\label{fig:fplot}}\qquad
\subfloat[]{\includegraphics{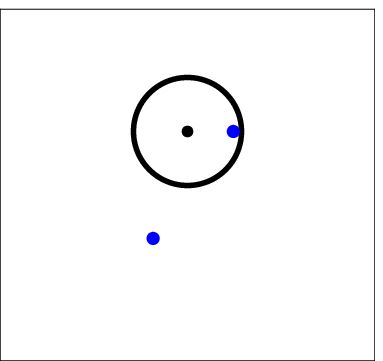}\label{fig:fplotfalse}}\qquad
\subfloat[]{\includegraphics{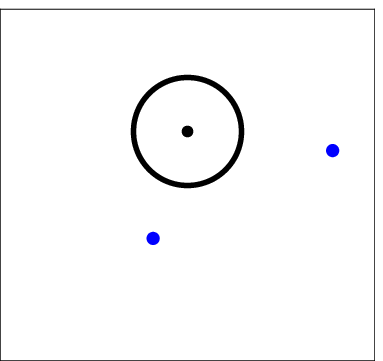}\label{fig:fplottrue}}
\caption{}
\end{figure}
The two points $p, q$ (in blue) in Figure~\eqref{fig:fplotfalse} can not be
connected via a continuous path in $\{f \neq 0\}$, since they belong to
different connected components. Hence the output should be {\tt false}. The
two points $p, q$ (in blue) in Figure~\eqref{fig:fplottrue}, however, can be connected via a continuous path in $\{f \neq 0\}$, since they belong to a same connected component. Hence the output should be {\tt true}. 
\end{example}

We present a symbolic-numeric algorithm which solves Problem~\eqref{problem}.   The algorithm will be described in two papers.  In this paper we focus on the symbolic part of the algorithm and assume the availability of  the numeric part which will be described in the forthcoming second  paper. We chose to divide the description into two papers because the methods and techniques are quite different from each other and interesting on their own. 

The paper is organized as follows. In Section~\ref{sec:algorithm} we will describe an algorithm for tackling the connectivity problem. It first searches for a certain ``nice'' function  $g$ based on $f$ and then tries to connect the input points via several steepest ascent paths of $ g$.
 In Section~\ref{sec:correctness} we prove the correctness of the algorithm assuming that a ``nice''  function $g$ is found,  by adapting Morse theory. In Section~\ref{sec:termination} we prove the termination of the algorithm by showing  that a ``nice'' function $g$ can be found in finite time, by using  Sard's Theorem and the Constant Rank Theorem.  In Section~\ref{examples} we give several non-trivial examples illustrating the use of the method. In Section~\ref{conclusion} we summarize the results in the paper and briefly mention future works.

 
\section{Algorithm}\label{sec:algorithm}

In this section, we describe a symbolic-numeric algorithm called
\textbf{Connectivity}. The algorithm \textbf{Connectivity} has the
same input and output specification we introduced in
Problem~\ref{problem}. This section is divided into two subsections. The
first subsection describes only the input/output specification of a
certified numeric subalgorithm called \textbf{Destination}, whose
steps will be described in the forthcoming second paper. The second
subsection describes the steps of the algorithm \textbf{Connectivity},
which relies on the use of \textbf{Destination}.  \new{Given a $C^1$
  map $\phi: \mathbb{R}\to\mathbb{R}^n$, we denote by $\phi'$ its
  derivative.}
\subsection{Specification of the Numeric Subalgorithm \textbf{Destination} }

In this subsection, we will describe the input/output specification of a certified numeric
subalgorithm called \textbf{Destination}, whose steps will be described in the forthcoming
second paper. We begin by introducing some definitions. Throughout this paper we let $\lVert \cdot \rVert$ denote the Euclidean norm and for a non-zero vector $v$, we let $\widehat{v} = \frac{v}{\lVert v \rVert}$.

\begin{definition}\label{def:trajectory} Let $g \colon \RR^n \to \RR$ be a $C^2$ function. Let $p$ be a point in $\RR^n$ and $v$ be a unit vector in $\RR^n$.  We say $\phi$ is a \textit{trajectory of $\nabla g$ through $p$ using $v$} if $\phi \colon \RR_+ \to \RR^n$ is a $C^2$ function and
\begin{equation}\label{eq:phi}
\forall t > 0 \left(\phi'(t) = \nabla g\bigl(\phi(t)\bigr) \text{ and } \phi'(t) \neq 0 \right)
\end{equation}
and 
\[
\lim_{t \to 0^+} \phi(t)  = p
\]
and 
\[
\lim_{t \to 0^+} \frac{\phi'(t)}{\lVert \phi'(t) \rVert} = \lim_{t \to 0^+} \frac{\nabla g\bigl(\phi(t)\bigr)}{\left\lVert \nabla g\bigl(\phi(t)\bigr) \right\rVert} = v.
\]

We call the image $\phi\bigl((0,\infty)\bigr)$ a \textit{steepest ascent path through $p$ using $v$} and denote this as $\SA(g, p, v)$. We call $\dest(\phi)$ a \textit{destination of $\phi$} if the following limit exists:
\[
\dest(\phi) = \lim_{t\to \infty} \phi(t).
\]
We say a point \textit{$q \in \RR^n$ is reachable from $p$ using $g$ and $v$} if there exists $\phi$, a trajectory of $\nabla g$ through $p$ using $v$, such that $\dest(\phi) = q$.
\end{definition}

\begin{example}\label{ex:traj} Let 
\begin{equation}\label{eq:g}
g = \frac{\left(-2 x_1^2 + x_1^4 - 2 x_2^2 + 2 x_1^2 x_2^2 + x_2^4\right)^2}{\bigl(x_1^2 + (x_2-1)^2 + 1\bigr)^5},
\end{equation}
$v = \langle -1, 0\rangle$, and $r_2$ be a critical point of $g$ shown in
Figure~\eqref{fig:reachable}. There exists a trajectory $\phi$ of $\nabla
g$ through $r_2$ using $v$ satisfying the properties above. In Figure~\eqref{fig:reachable} we illustrate $\SA(g, r_2, v)$ as the red curve, where $v$ is shown as the green arrow. We see the point $r_4$ is reachable from $r_2$ using $g$ and $v$ because $\dest(\phi) = r_4$.
\begin{figure}[ht]
\centering
\includegraphics{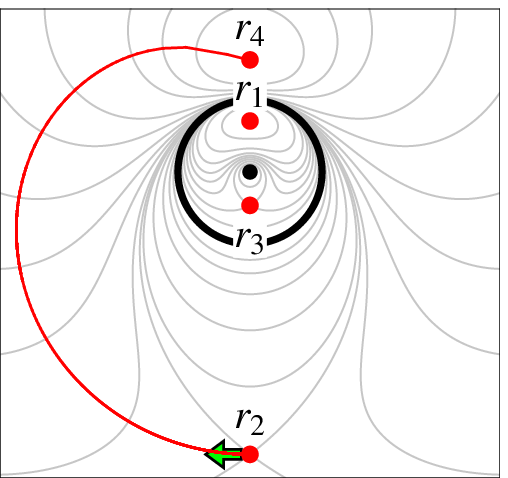} 
\caption{}
\label{fig:reachable}
\end{figure} 
\end{example}

\arraycolsep=2pt\def\arraystretch{1}

\begin{figure}
\begin{algorithm}[H]
\Indentp{-1.5em}
\textsc{Algorithm}: $t \gets$ \textbf{Connectivity}$(f, p, q)$\;
\Input{$f \in \mathbb{Z}[x_1,\ldots,x_n]$, $n \geq 2$, $\deg f \geq 1$, squarefree, with finitely many singular points,\\
$p, q \in \mathbb{Q}^n \cap \{f \neq 0\}$.
}
\Output{$t$,  {\tt true} if the two points $p$ and $q$
  lie in a same semi-algebraically connected component of $\{f \neq 0\}$, else {\tt false}.}
\Indentp{1.5em}

\BlankLine

\nlset{1}
$\begin{array}[t]{rl}
\gamma &\gets \deg(f) + 1\\
c & \gets (0, \dotsc, 0)
\end{array}$\;

\nlset{2}
\Loop{}{
$\begin{array}[t]{rl}
U &\gets (x_1 - c_1)^2 + \dotsb + (x_n - c_n)^2 + 1\\
\mathcal{F} &\gets \bigl\{2 \cdot (\partial_{x_i} f) \cdot U - \gamma \cdot f \cdot (\partial_{x_i}U)\bigr\}_{i=1}^n\\
g &\gets {\displaystyle \frac{f^2}{U^\gamma}}
\end{array}$\;
\lIf{$\left(\begin{array}{l} V_{\new{\mathbb{C}}}(\mathcal{F}) \text{ is zero-dimensional and}\\
\forall r \in V_{\new{\mathbb{C}}}(\mathcal{F}), g(r) \neq 0  \implies \det (\Hess g)(r) \neq 0\end{array}\right)$}{
exit \LOOP
} 
\ELSE $c \gets $ perturb current $c$ on the nonnegative integer grid
} 

\nlset{3}
$R \gets V_{\new{\mathbb{R}}}(\mathcal{F}) \setminus V_{\new{\mathbb{R}}}\bigl(f\bigr)$\;

\nlset{4}
$\begin{array}[t]{rl}
A & \gets k \times k \text{ matrix with all entries set to 0,}\\
  & \quad \text{where $k$ is the number of points in $R$}

\end{array}$\;

\nlset{5}
\ForEach{$r \in R$}{
$\begin{array}[t]{rl}
V_r & \gets \text{set of real algebraic orthonormal eigenvectors of $(\Hess g)(r)$ having}\\
&\quad \text{positive eigenvalues and whose first coordinate is non-positive}
\end{array}$\;

  \ForEach{$v \in V_r$}{

  $\begin{array}[t]{rl} 
j_+ & \gets \text{\textbf{Destination}}(g, R, r_i, +v)\\
j_- & \gets \text{\textbf{Destination}}(g, R, r_i, -v)\\
A_{ij_+} & \gets 1\\
A_{ij_-} & \gets 1
\end{array}$
} 
} 

\nlset{6}
$\begin{array}[t]{rl}
M & \gets \text{the reflexive, symmetric and transitive closure of the relation}\\
 & \quad \; \text{represented by the matrix $A$}
\end{array}$\;

\nlset{7}
\lIf{$\nabla g(p) \neq 0$}{$i \gets \text{\textbf{Destination}}\left(g, R, p, \widehat{\nabla g(p)}\right)$}
\lElse{$i \gets$ index of $p$ in $R$}

\nlset{8}
\lIf{$\nabla g(q) \neq 0$}{$j \gets \text{\textbf{Destination}}\left(g, R, q, \widehat{\nabla g(q)}\right)$}
\lElse{$j \gets$ index of $q$ in $R$}

\nlset{9}
\Return $t \gets$ \texttt{true} if $M_{ij} = 1$, else $t \gets$ \texttt{false};

\BlankLine

\Indentp{-1.5em}
\hrulefill

\BlankLine
\BlankLine

\textsc{Algorithm}: $i \gets$ \textbf{Destination}$(g, R, p, v)$\;
\Input{$g \colon \mathbb{R}^n \to \mathbb{R}$, $C^2$ function,\\
$R = \{R_1, \dotsc, R_k\}$, each $R_i \in \mathbb{A}^n$,\\
$p \in \mathbb{A}^n$,\\
$v \in \mathbb{A}^n$, $\lVert v \rVert = 1$, \\
 such that there exists a unique $r \in R$
 reachable from $p$ using $g$ and $v$.
}
\Output{$i$, the positive integer such that $r = R_i$.
}
\end{algorithm}
\caption{}\label{fig:algorithms}
\end{figure}

We state the specification for the algorithm \textbf{Destination} in
Figure~\eqref{fig:algorithms} and give an example input and output in the following example. Note that throughout this paper we will denote the set of algebraic numbers by $\mathbb{A}$.

\begin{example} Let $g$ be as in \eqref{eq:g}, $R = \{r_1, r_2, r_3, r_4\}$
    be the list of points in red, and $v$ be the vector shown as the green
    arrow in Figure~\eqref{fig:reachable}, respectively. Let $p = r_2$. The point $r_4$ is the unique point that is reachable from $r_2$ using $g$ and $v$. Hence the output of \textbf{Destination}$(g, R, p, v)$ would be 4.
\end{example}

\subsection{Description of Algorithm \textbf{Connectivity}}

In this subsection, we will illustrate the steps of the algorithm
\textbf{Connectivity} using the toy problem given in
Example~\ref{ex:toy}. We will provide several pictures in the hope of
aiding intuitive understanding of what each step does. Of course, the
algorithms do not draw the pictures. We state the steps of
\textbf{Connectivity} in Figure~\eqref{fig:algorithms}. We use the
following notations. For a family $\mathcal{F} = \{f_1, \dotsc, f_n\}$
of polynomials in $\ZZ[x_1, \dotsc, x_n]$, we let
$V_\mathbb{R}(\mathcal{F})$ \new{(resp. $V_\mathbb{C}(\mathcal{F})$)}
denote the zero-locus in $\mathbb{R}^n$ \new{(resp. $\mathbb{C}^n$)}
of the polynomials in $\mathcal{F}$. For a $C^2$ function $g$ we let
$\Hess g$ denote the Hessian matrix of $g$.

\begin{remark} The algorithm \textbf{Connectivity} consists of three main stages.
\begin{enumerate}
\item Using $f$, compute ``interesting'' points on each connected component of $\{f \neq 0\}$. Create a function $g$ with desirable properties, one being that $g = 0$ if and only if $f = 0$. Use $g$ and the ``interesting'' points to form some vectors. 
\item Connect the ``interesting'' points on each connected component of $\{g \neq 0\}$ using the vectors and trajectories of $\nabla g$ to create a connectivity matrix $M$ by using \textbf{Destination}.
\item Determine the connectivity of $p$ and $q$ using $M$ and trajectories of $\nabla g$, again using \textbf{Destination}.
\end{enumerate}
The first and second stage are much more time-consuming than the third one. Fortunately, one needs to carry out the first and second stage {\em only once\/} for a given $f$, since it  depends only on $f$.
\end{remark}

\begin{example}~ We give a sample run of \textbf{Connectivity} using the toy example from Example~\ref{ex:toy}.

\begin{glists}{0em}{1em}{1em}{0.3em}

\itemr{\textbf{Input}.} $f = -2 x_1^2 + x_1^4 - 2 x_2^2 + 2 x_1^2 x_2^2 +
    x_2^4$, $p = \left(19/5, -1/2\right), q = \left(-9/10, -14/5\right)$
    are the blue points in Figure~\eqref{fig:fplottrue}.
\begin{itemize}
\item Here, $n = 2$, $\deg f = 4$, and $f$ is a squarefree polynomial with exactly one singular point at $(0,0)$.
\end{itemize} 
\itemr{1.} Initially, we have 
\begin{align*}
\gamma &= 5,\\
c &= (0, 0).
\end{align*}

\itemr{2.} In the first iteration of the loop we have
\begin{align*}
U &= x_1^2 + x_2^2 + 1,\\
\mathcal{F} &= \left\{-2 x_1^5-4 x_2^2 x_1^3+20 x_1^3-2 x_2^4 x_1+20 x_2^2 x_1-8 x_1,\right.\\
& \qquad \left.-2 x_2^5-4 x_1^2 x_2^3+20 x_2^3-2 x_1^4 x_2+20 x_1^2 x_2-8 x_2\right\},\\
g &=  \frac{\left(-2 x_1^2 + x_1^4 - 2 x_2^2 + 2 x_1^2 x_2^2 + x_2^4\right)^2}{\left(x_1^2+x_2^2+1\right)^5}. 
\end{align*}

The current $V_{\new{\mathbb{C}}}(\mathcal{F})$ is one-dimensional. In
Figure~\eqref{fig:unperturbedwithg} we illustrate the contours for the current $g$ in gray and $V_{\new{\mathbb{R}}}(\mathcal{F})$ in red. 
\begin{figure}[ht]
                                \centering
                                \subfloat[]{\includegraphics{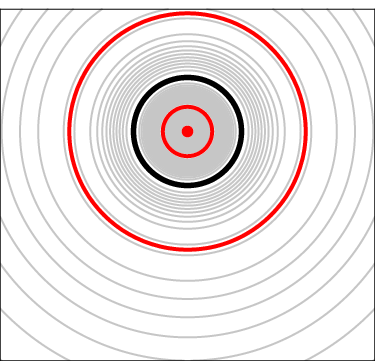}\label{fig:unperturbedwithg}} \qquad
                                \subfloat[]{\includegraphics{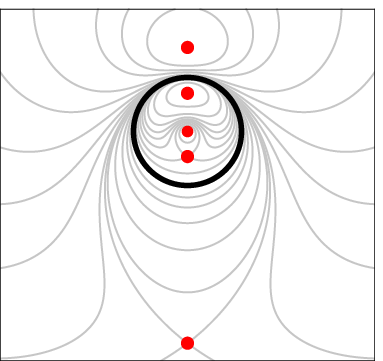}\label{fig:perturbedwithg}} \qquad
                                \subfloat[]{\includegraphics{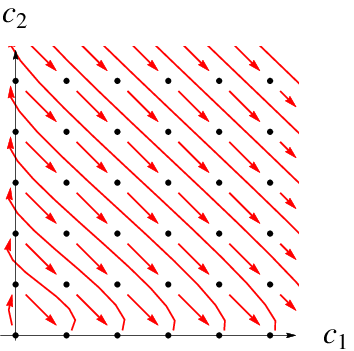}\label{fig:glexorder}}
                                \caption{}                                                      
                          \end{figure}  
We perturb $c$ on the integer grid to be $c = (0,1)$. In the second iteration of the loop we update $U$, $\mathcal{F}$ and $g$ to be
\begin{align*}
U &= x_1^2 + (x_2 - 1)^2 + 1,\\
\mathcal{F} &= \left\{-2 x_1^5-4 x_2^2 x_1^3-16 x_2 x_1^3+28 x_1^3-2 x_2^4 x_1\right.\\
& \qquad -16 x_2^3 x_1+28 x_2^2 x_1+16 x_2 x_1-16 x_1,\\
& \qquad -2 x_2^5-6 x_2^4-4 x_1^2 x_2^3+28 x_2^3+4 x_1^2 x_2^2-4 x_2^2\\
& \qquad \left.-2 x_1^4 x_2+28 x_1^2 x_2-16 x_2+10 x_1^4-20 x_1^2\right\},\\
g &=\frac{\left(-2 x_1^2 + x_1^4 - 2 x_2^2 + 2 x_1^2 x_2^2 + x_2^4\right)^2}{\left(x_1^2+(x_2 - 1)^2+1\right)^5}.
\end{align*}
The new $V_{\new{\mathbb{C}}}(\mathcal{F})$ is zero-dimensional (over
the reals).  We illustrate the perturbed
$V_{\new{\mathbb{R}}}(\mathcal{F})$ as the five red points in
Figure~\eqref{fig:perturbedwithg} along with the contours for the new
$g$. For all five $r \in V_{\new{\mathbb{C}}}(\mathcal{F})$, $\det (\Hess g)(r) \neq 0$
Hence we exit the loop.
\begin{itemize}             
\item One method for perturbing is using graded lexicographic
  order. In Figure~\eqref{fig:glexorder}, if there is an arrow having
  tip at $\alpha$ and tail at $\beta$ then $x^\alpha > x^\beta$ in the
  graded lexicographic order. We can follow the arrows to
  systematically change $(c_1, c_2)$ starting at $(0,0)$. This
  generalizes, of course, to any number of variables.
\item \new{In this step we can use symbolic computation methods based
  e.g. on Gr\"obner bases 
  to compute the dimension of the zero-locus of $\mathcal{F}$ in
  $\mathbb{C}^n$ (see e.g.
  \cite{CoxLittleOShea}).} 
\end{itemize}

\itemr{3.} We illustrate $R$ as the four red points in
    Figure~\eqref{fig:critscontoursg}. Compare this to the five red points
    in Figure~\eqref{fig:perturbedwithg}.
                                                \begin{figure}[ht]
                                \centering
                                                                        \includegraphics{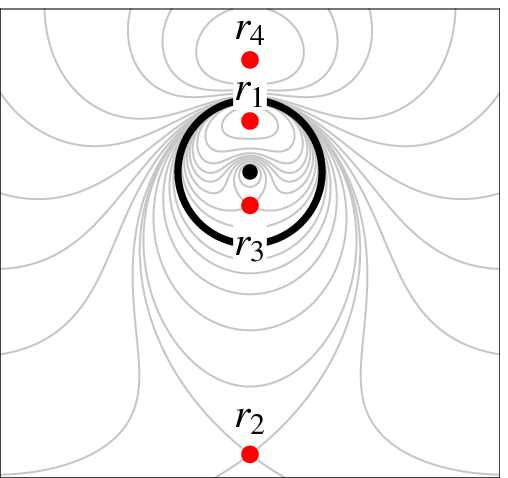}
                                \caption{}
                                \label{fig:critscontoursg}
                          \end{figure}  
\begin{itemize}
\item Note that each connected component of $\{f \neq 0\}$ contains at
  least one point from the set $R$.
\item One may observe from the contour plot of $g$, that the points
  $R$ are critical points of $g$ where $g$ is non-zero.
\item Again, we can use standard symbolic computation methods to
  identify which of the points in $V_\mathbb{R}(\mathcal{F})$ satisfy $f = 0$,
  and then remove them. 

  \new{This can be done by representing $V_\mathbb{C}(\mathcal{F})$ by
    means of the so-called rational univariate representation (see
    e.g. \cite{ABRW96,GiMo89,Rou99}) and Thom-encodings for the real
    roots (see e.g. \cite{BaPoRo06}).}
\end{itemize}

\itemr{4.} We have $A = \begin{bmatrix}0 & 0 & 0 & 0\\ 0 & 0 & 0 & 0\\ 0 & 0 & 0 & 0\\ 0 & 0 & 0 & 0\end{bmatrix}$ since $k = 4$.

\itemr{5.} Suppose $r = r_1$ or $r = r_4$. The matrix $(\Hess g)(r)$ has no positive eigenvalues. Hence $V_r = \emptyset$ and the body of the second \textbf{foreach} loop does not execute. 

Suppose instead that $r=r_2$ or $r = r_3$, then the matrix $(\Hess g)(r)$
has one positive eigenvalue. For this eigenvalue, there is one
corresponding real algebraic unit eigenvector whose first coordinate is
non-positive. If $r = r_2$, the eigenvector is $v = \langle -1, 0\rangle$.
We draw the two vectors $v$ and $-v$ as a dark green and light green
outward pointing arrows from $r_2$ in Figure~\eqref{fig:fatarrows},
respectively. If $r = r_3$, the eigenvector is $v = \langle -1, 0\rangle$.
We draw the two vectors $v$ and $-v$ as a dark blue and light blue outward
pointing arrows from $r_3$ in Figure~\eqref{fig:fatarrows}, respectively.

\begin{figure}[ht]
\centering
\subfloat[]{\includegraphics{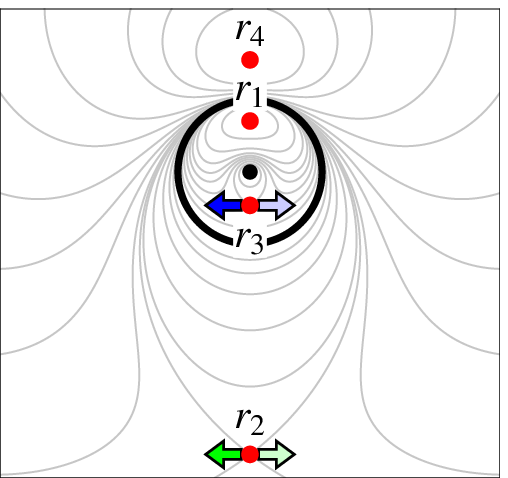}\label{fig:fatarrows}} \qquad
\subfloat[]{\includegraphics{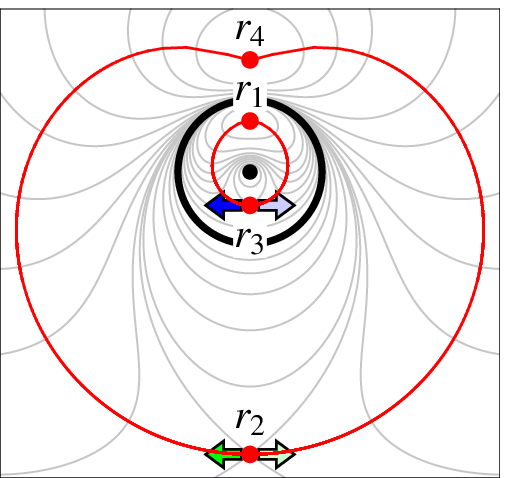}\label{fig:fatarrowsroadmap}}
\caption{}
\end{figure}  
We let $V_{r_2}$ be the set consisting of a single unit eigenvector, $\langle -1, 0\rangle$, represented by the dark green arrow, and $V_{r_3}$ be the set consisting of a single unit eigenvector, $\langle -1, 0\rangle$, represented by the dark blue arrow.

Figure~\eqref{fig:fatarrowsroadmap} shows four steepest ascent paths. The steepest ascent path starting from the point $r_2$ in the direction of the dark green vector approaches the point $r_4$. Similarly, the steepest ascent path starting from the point $r_2$ in the direction of the light green vector approaches the point $r_4$. Hence when $r = r_2$, the inner \textbf{foreach} loop executes once because there is only one vector in $V_{r_2}$ and 
\begin{align*}
j_+ &\gets \text{\textbf{Destination}}\left(g, R, r_2, \includegraphics{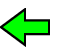}\right) = 4,\\
j_- &\gets \text{\textbf{Destination}}\left(g, R, r_2, \includegraphics{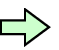}\right) = 4,\\
A_{24} &\gets 1,\\
A_{24} &\gets 1.
\end{align*}
The steepest ascent path starting from the point $r_3$ in the direction of the dark blue vector approaches the point $r_1$. Similarly, The steepest ascent path starting from the point $r_3$ in the direction of the light blue vector approaches the point $r_1$. Hence when $r = r_3$, the inner \textbf{foreach} loop executes once because there is only one vector in $V_{r_3}$ and 
\begin{align*}
j_+ &\gets \text{\textbf{Destination}}\left(g, R, r_3, \includegraphics{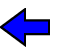}\right) = 1,\\
j_- &\gets \text{\textbf{Destination}}\left(g, R, r_3, \includegraphics{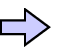}\right) = 1,\\
A_{31} &\gets 1,\\
A_{31} &\gets 1.
\end{align*}
The matrix $A$ has the form
\[ 
A = \begin{bmatrix}
0 & 0 & 0 & 0\\
0 & 0 & 0 & 1\\
1 & 0 & 0 & 0\\
0 & 0 & 0 & 0
\end{bmatrix}.
\]
\begin{itemize}
\item For each $r \in R$, the Hessian $(\Hess g)(r)$ is a real
  symmetric matrix. It is a well known fact that the associated
  eigenvalues are all real and the eigenvectors corresponding to
  different eigenvalues are orthogonal. However, there is no
  restriction that the eigenvalues be simple, so it is possible that
  the geometric multiplicity of a positive eigenvalue is greater than
  one. In this case, finding two linearly independent eigenvectors for
  a given positive eigenvalue will suffice, as one can use the
  Gram-Schmidt process to find an orthonormal basis.
\item Using standard symbolic computation techniques, we can find the
  eigenvalues and eigenvectors exactly because each point in $R$ is an
  algebraic number and the elements of $\Hess g$ are rational
  functions with integer coefficients. 

  \new{This can be done by e.g. solving the system defined by the vanishing
  of the polynomials in ${\cal F}$, the equations
  $\Hess g . V =\lambda V$ and $V\neq \mathbf{0}$ where the entries
  $\lambda$ are new variables. When $V_{\mathbb{C}}({\cal F})$ has
  dimension zero and does not meet the hypersurface defined by
  $\det(\Hess g)=0$, this system has dimension $0$ also.}

\item Note that every steepest ascent path approaches a point in the
  set $R$.  In fact, $g$ was constructed to ensure that the path never
  spirals in a bounded region or goes forever into the infinity.
\item It is crucial to observe that every two points in $R$ can be
  connected if and only if they are connected via the above computed
  paths.
\end{itemize}

\itemr{6.} We have $M = 
\begin{bmatrix}
1 & 0 & 1 & 0\\
0 & 1 & 0 & 1\\
1 & 0 & 1 & 0\\
0 & 1 & 0 & 1
\end{bmatrix}$.
\begin{itemize}
           \item 
               Note that we can use the matrix $M$ to check whether two points $r_i, r_j \in R$ lie in a same connected component of $\{f \neq 0\}$ by checking the $(i,j)$ entry of $M$. 
           \item We call $M$ a \textit{connectivity matrix}.
           \end{itemize}      
 
\itemr{7.} For the input point $p$ shown in Figure~\eqref{fig:roadmaptrue}, $\nabla g(p) \neq 0$.
\begin{figure}[ht]
\centering
\includegraphics{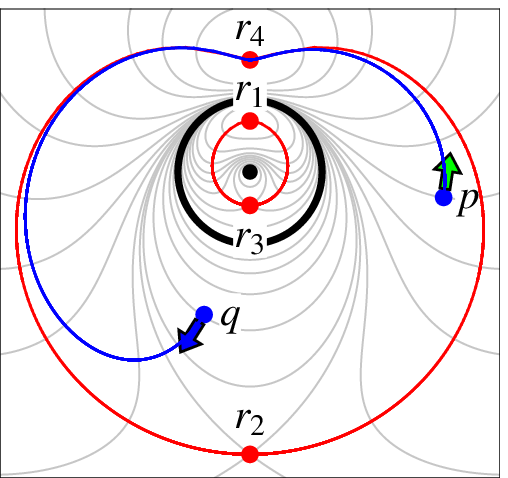}
\caption{}
\label{fig:roadmaptrue}
\end{figure} 
We draw the vector $\widehat{\nabla g(p)}$ as the green arrow. We see that steepest ascent from $p$ approaches the point $r_4$ in $R$. Hence 
\[
i \gets \text{\textbf{Destination}}\left(g, R, p, \includegraphics{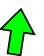}\right) = 4.
\]

\itemr{8.} For the input point $q$ shown in Figure~\eqref{fig:roadmaptrue}, $\nabla g(q) \neq 0$. We draw the vector $\widehat{\nabla g(q)}$ as the blue arrow. We see that steepest ascent from $q$ approaches the point $r_4$ in $R$. Hence 
\[
j \gets \text{\textbf{Destination}}\left(g, R, q, \includegraphics{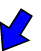}\right) = 4.
\]

\itemr{9.} We note that $M_{44} = 1$ and thus the two points $p$, $q$ can be connected. We set $t$ = \texttt{true}.

\itemr{\textbf{Output}.} $t = \texttt{true}$.
\end{glists}
\end{example}


\section{Correctness}\label{sec:correctness}

In this section, we will prove the correctness of the algorithm \textbf{Connectivity} in the form of Theorem~\ref{thm:correct}. 

\begin{theorem}\label{thm:correct}
Algorithm \textbf{Connectivity} is correct.
\end{theorem}

\noindent It essentially amounts to showing that any two ``interesting'' points in the same connected component of $\{g \neq 0\}$ are connected by a particular set of steepest ascent paths. In order to make the claim precise, we will need to recall and introduce some notations and notions. We assume throughout this section that $g \colon \RR^n \to \RR$ is a $C^2$ function with $n \geq 2$. The examples in this section will assume $g$ takes the form \eqref{eq:g}.

\begin{definition}\label{def:routing} A critical point $p$ of $g$ is called a \textit{routing point} of $g$ if $g(p) \neq 0$. Let $R$ be the set of routing points of $g$. We call $g$ a \textit{routing function} if the following conditions are satisfied:
\begin{itemize}
\item For all $x$, $g(x) \geq 0$.
\item For all $\varepsilon > 0$, there exists $\delta > 0$, such that for all $x$, $\lVert x \rVert \geq \delta$ implies $g(x) \leq \varepsilon$.
\item $R$ is finite.
\item For all $x \in R$, $x$ is nondegenerate; that is, $(\Hess g)(x)$ is non-singular.
\item The norms of the first and second derivatives of $g$ are bounded.
\end{itemize}
\end{definition}

Intuitively, the second condition in the routing function definition says that $g$ vanishes at infinity; that is, as $\lVert x \rVert \to \infty$, then $g(x) \to 0$.

\begin{example}In Figure~\eqref{fig:critscontoursg} we show the contours of $g$ along with the routing points of $g$ as red dots. The black curve and black dot is the set of points where $g = 0$. One may easily check that $g$ in \eqref{eq:g} satisfies the conditions to show $g$ is a routing function. 
\end{example}

In our forthcoming examples, we will let $R = \{r_1, r_2, r_3, r_4\}$ denote the set of routing points of $g$ from \eqref{eq:g}.

\begin{definition}
Let $A$ be a real symmetric matrix and let $v$ be a unit eigenvector of $A$ with corresponding eigenvalue $\lambda \neq 0$. We say $v$ is an \textit{outgoing eigenvector} if $\lambda > 0$.
\end{definition}

\begin{example} 
In Figure~\eqref{fig:fatarrows}, outgoing eigenvectors of $(\Hess g)(r_i)$ are shown as arrows pointing outward from the point $r_i$.
\end{example}

\begin{definition}\label{def:connsteepasc} 
Let $p \neq q$ be two points in $\mathbb{R}^n$. We say $p$ and $q$ are \textit{connected by steepest ascent paths using outgoing eigenvectors of $g$} if there exist functions $\phi_1, \dotsc, \phi_{k+1}$ and routing points $r_1, \dotsc, r_k$ of $g$ such that 
\begin{itemize}
\item if $\nabla g(p) = 0$, then $\phi_1 = p$ and $r_1 = p$, otherwise, $\phi_1$ is a trajectory of $\nabla g$ through $p$ using $\widehat{\nabla g(p)}$ and $\dest(\phi_1) = r_1$,
\item if $\nabla g(q) = 0$, then $\phi_{k+1} = q$ and $r_k = q$, otherwise, $\phi_{k+1}$ is a trajectory of $\nabla g$ through $q$ using $\widehat{\nabla g(q)}$ and $\dest(\phi_{k+1}) = r_k$,
\item for all $2 \leq i \leq k$, there exists an outgoing eigenvector $v_{i-1}$ of $(\Hess g)(r_{i-1})$ such that $\phi_i$ is a trajectory of $\nabla g$ through $r_{i-1}$ using $v_{i-1}$ and $\dest(\phi_i) = r_i$, or, there exists an outgoing eigenvector $v_i$ of $(\Hess g)(r_i)$ such that $\phi_i$ is a trajectory of $\nabla g$ through $r_i$ using $v_i$ and $\dest(\phi_i) = r_{i-1}$.
\end{itemize}
Collectively, we call $r_1, \dotsc, r_k$ and $\phi_1, \dotsc, \phi_{k+1}$ a \textit{connectivity path for $p$ and $q$}.
\end{definition}

\begin{example} In Figure~\eqref{fig:reachable}, we see $r_2$ is connected to $r_4$ by steepest ascent paths using outgoing eigenvectors of $g$. This is because there exist routing points $r_2, r_4$ and functions $\phi_1 = r_2$, $\phi_3 = r_4$, and $\phi_2$, a trajectory of $\nabla g$ through $r_2$ using $v$, an outgoing eigenvector of $(\Hess g)(r_2)$, such that $\dest(\phi_2) = r_4$. 
\end{example}

\begin{example} In Figure~\eqref{fig:roadmaptrue}, we see $p$ is connected to $q$ by steepest ascent paths using outgoing eigenvectors of $g$. This is because there exists a routing point $r_4$ and functions $\phi_1$, $\phi_2$, such that $\phi_1$ is a trajectory of $\nabla g$ through $p$ using $\widehat{\nabla g(p)}$ and $\dest(\phi_1) = r_4$, and $\phi_2$ is a trajectory of $\nabla g$ through $q$ using $\widehat{\nabla g(q)}$ and $\dest(\phi_2) = r_4$.
\end{example}

We now state the theorem we will use to help us prove the correctness of the algorithm \textbf{Connectivity}. 

\begin{theorem}\label{thm:connresult} If $g$ is a routing function then any two points in a same connected component of $\{g \neq 0\}$ are connected by steepest ascent paths using outgoing eigenvectors of $g$.
\end{theorem}

To prove Theorem~\ref{thm:connresult}, we will use results motivated from the field of Morse theory. In Morse theory, one analyzes the topology of a manifold by studying differentiable functions on that manifold. In our case, we will be studying the manifold $\RR^n$ and decomposing a region into sets of similar behavior based on trajectories. 

\begin{definition} If $p \in \RR^n$ is a nondegenerate critical point of $g$, then the \textit{stable manifold of $p$} is defined to be 
\[
W^s(p) = \set{x \in \RR^n}{\dest\left(\phi_x\right) = p} \cup \{p\}. 
\]
where $\phi_x$ is the trajectory of $\nabla g$ through $x$ using $\widehat{\nabla g(x)}$. 
\end{definition}

Similar definitions can be found in the literature about Morse-Smale gradient fields \cite[Definition 4.1, pp. 94]{Banyaga2004}. It is important to note that the fact a stable manifold is actually a manifold (in the sense of differential geometry) is non-trivial and relies on the genericity of $g$ \cite[Section 4.1]{Banyaga2004}. Generally, stable manifolds cannot be found analytically. Instead, these manifolds must be ``grown'' from the fixed point $p$ using local knowledge \cite{KrauskopfEtAl2005}.

\begin{example}\label{ex:stunstmani} Figure~\eqref{fig:stablemans} illustrates the stable manifolds for the routing points of $g$. The stable manifolds for $r_1$ and $r_4$ are the blue and green regions, respectively. The stable manifold for $r_2$ is the blue line while the stable manifold for $r_3$ is the green line. 
\begin{figure}[ht]
\centering
\includegraphics{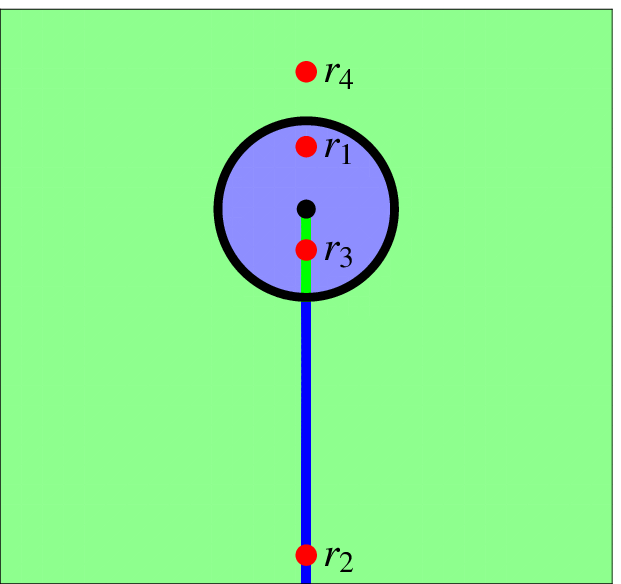}
\caption{}
\label{fig:stablemans}
\end{figure}
\end{example}

According to Figure~\eqref{fig:stablemans}, it appears we can decompose each connected component of $\{g \neq 0\}$ into a disjoint union of stable manifolds. We will use the following lemmas to show that if $g$ is a routing function, we can in fact decompose a connected component into a disjoint union of stable manifolds. First, we observe the simple fact that $g$ strictly increases along a steepest ascent path.

\begin{lemma}\label{lem:trajincr} Let $p \in \RR^n$. If $p$ is not a critical point of $g$ then $g$ increases along a trajectory of $\nabla g$ through $p$ using $\widehat{\nabla g(p)}$. 
\end{lemma}

\begin{proof} Let $p \in \RR^n$ with $\nabla g(p) \neq 0$. Let $\phi$ denote a trajectory of $\nabla g$ through $p$ using $\widehat{\nabla g(p)}$. Since $g$ is $C^2$, we know $\phi$ is $C^1$ and hence $g \circ \phi$ is $C^1$. We compute
\begin{equation}\label{eq:dRHS}
\frac{d}{dt} g\bigl(\phi(t)\bigr) = \left\langle \nabla g\bigl(\phi(t)\bigr), \phi'(t)\right\rangle=\left\langle \nabla g\bigl(\phi(t)\bigr), \nabla g\bigl(\phi(t)\bigr)\right\rangle = \left\lVert \nabla g\bigl(\phi(t)\bigr) \right\rVert^2.
\end{equation}
Since $p$ is not a critical point of $g$, $\left\lVert\nabla g\bigl(\phi(t)\bigr) \right\rVert^2 > 0$ for all $t > 0$. It follows from \eqref{eq:dRHS} that
\[
\frac{d}{dt} g\bigl(\phi(t)\bigr) > 0
\]
for all $t > 0$. Hence $g$ strictly increases along $\phi$.
\end{proof}

For the rest of the section we assume $g$ is a routing function and $D$ is a connected component of $\{g \neq 0\}$. First, we state some simple facts about $g$.

\begin{lemma} \label{lem:gbounded}
$g$ is a bounded over the reals.
\end{lemma}

\begin{proof}
The first property in Definition~\ref{def:routing} guarantees $g$ is bounded below by 0. Suppose for a contradiction that $g$ is not bounded above. Then for all $M$, there exists $x \in \RR^n$ such that $\lvert g(x) \rvert > M$. In particular, for every $k \in \NN$, there exists $x_k \in \RR^n$ for which $\lvert g(x_k) \rvert > k$. Fix such a sequence $\{x_k\}_{k=1}^\infty$. Certainly, $g(x_k) \geq 0$ for all $k$. Let
\begin{align*}
L &= \min_{k \in \NN} \set{g(x_k)}{g(x_k) > 0},\\
S &= \set{x \in \RR^n}{g(x) \geq L}.
\end{align*}
Let $k^*$ be the index such that $g(x_{k^*}) = L$. The second property in Definition~\ref{def:routing} guarantees $S$ is bounded by letting $\varepsilon = L > 0$. Since the tail $\{x_k\}_{k=k^*}^\infty$ is contained in $S$, the Bolzano-Weierstrass theorem implies there exists a subsequence $\{x_{k_j}\}_{j=1}^\infty$ which converges to some limit $M$. Since $g$ is continuous everywhere, 
\[
\lim_{j \to \infty} g(x_{k_j}) = g(M).
\]
In particular, the sequence $\bigl\{g(x_{k_j})\bigr\}_{j=1}^\infty$ is convergent, hence bounded. However, by construction, $\bigl \lvert g(x_{k_j}) \bigr \rvert > k_j \geq j$ for all $j \in \NN$, and hence this sequence is not bounded, a contradiction. Thus $g$ is bounded above. Therefore $g$ is bounded.
\end{proof}

\begin{example}
When we study the graph of $g$ from \eqref{eq:g} in Figure~\eqref{fig:plot3dgshaded}, we can see that $g$ is bounded over the reals between $[0, 2.185]$.
\begin{figure}[ht]
\centering
\subfloat[]{\raisebox{10mm}{\includegraphics{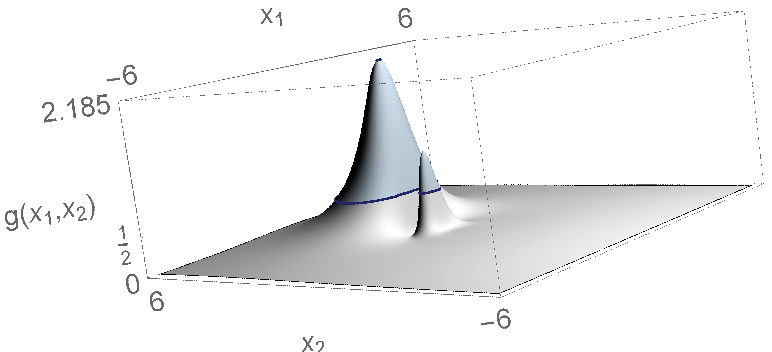}\label{fig:plot3dgshaded}}}\qquad
\subfloat[]{\includegraphics{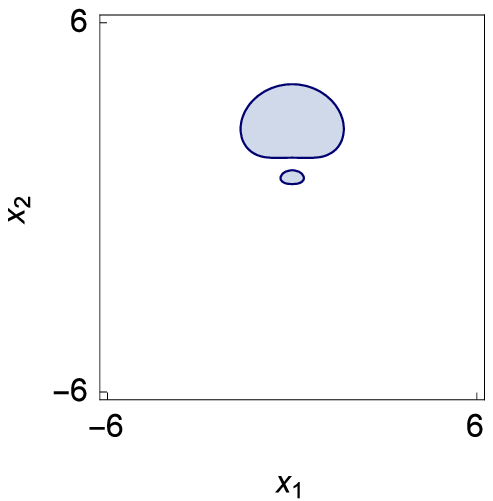}\label{fig:compact}}
\caption{}
\end{figure}
\end{example}

\begin{lemma} \label{lem:compactset}
For all $L > 0$, $\set{x \in D}{g(x) \geq L}$ is compact.
\end{lemma}

\begin{proof}
Let $L > 0$ and $K = \set{x \in D}{g(x) \geq L}$. Recall that $g$ is bounded (Lemma~\ref{lem:gbounded}), hence there exists $M$, such that for all $x \in \RR^n$, $\lvert g(x) \rvert \leq M$. The set 
\[
S = g^{-1}\bigl([L, M]\bigr) = \set{x \in \RR^n}{g(x) \geq L}
\]
is closed because it is the preimage of a closed set under a continuous function and bounded due to the second property of $g$ being a routing function (letting $\varepsilon = L > 0$). Therefore, $S$ is compact. The semi-algebraic set $S$ is a disjoint union of closed semi-algebraic connected components $K_1, \dotsc, K_\ell$ where $K_1 = K$ (without loss of generality). Since $K$ is a closed subset of the compact set $S$ then $K$ is compact.
\end{proof}

\begin{example} Consider the superlevel set $S = \left\{g \geq \frac{1}{2}
    \right\}$ of $g$ from \eqref{eq:g}. We draw $S$ as the shaded region in
    Figure~\eqref{fig:compact}. We see that $S$ consists of two connected
    components, each of which is compact. For reference, we also shaded the
    points on the graph of $g$ in Figure~\eqref{fig:plot3dgshaded} whose third coordinate is greater than or equal to $\frac{1}{2}$. 
\end{example}

We show next that the trajectories are unique assuming a certain condition.

\begin{lemma}  \label{lem:uniquephi}
Let $p \in D$. If $\nabla g(p) \neq 0$, then there exists a unique trajectory $\phi$ of $\nabla g$ through $p$ using $\widehat{\nabla g(p)}$
\end{lemma}

\begin{proof} 
Let $p \in D$. The component $D$ is an open subset of $\RR^n$ containing $p$ and $g \in C^2(D)$. According to the Fundamental Existence-Uniqueness Theorem  \cite[Section 2.2, pp. 74]{Perko2001}, there exists $a > 0$ such that
\begin{equation}\label{eq:IVP}
\begin{split}
\phi'(t) &= \nabla g\bigl(\phi(t)\bigr)\\
\phi(0) &= p.
\end{split} 
\end{equation}
has a unique solution $\phi(t)$ on the interval $[-a, a]$. Let $[0, \beta)$ be the right maximal interval of existence of $\phi(t)$.

Because $g$ is bounded (Lemma~\ref{lem:gbounded}), the trajectory $\phi$ is bounded. It follows from  \cite[Theorem 3, Section 2.4, pp. 91]{Perko2001} that $\beta = \infty$. Certainly $\lim_{t \to 0^+} \phi(t) = p$ and 
\[
\lim_{t \to 0^+} \frac{\phi'(t)}{\lVert \phi'(t) \rVert} = \widehat{\nabla g(p)}.
\]
Hence $\phi$ is the trajectory of $\nabla g$ through $p$ using $\widehat{\nabla g(p)}$.
\end{proof}

\begin{remark}\label{rem:uniquephi} A similar argument to the one above shows that if $p \in D$ and $\nabla g(p) \neq 0$ then there exists a unique $C^2$ function $\phi \colon (-\infty, 0] \to \RR^n$ satisfying 
\begin{align*}
\phi'(t) &= -\nabla g\bigl(\phi(t)\bigr)\\
\phi(0) &= p.
\end{align*} 
Combined with the argument above, this means that there exists a unique $C^2$ function $\phi \colon \RR \to \RR^n$ satisfying \eqref{eq:IVP}. When $\nabla g(p) = 0$, $\phi = p$ is the unique solution to \eqref{eq:IVP}, which exists for all $t \in \RR$. We can conclude that the gradient vector field $\nabla g$ is complete.
\end{remark}

Next, we have the important observation that the destination of every steepest ascent path is a routing point of $g$. 

\begin{lemma}\label{lem:origdestD} 
Let $p \in D$ with $\nabla g(p) \neq 0$ and $\phi$ be the trajectory of $\nabla g$ through $p$ using $\widehat{\nabla g(p)}$. Then $\dest(\phi)$ exists and is a routing point of $g$ in $D$. 
\end{lemma} 

\begin{proof}  
Let $p \in D$ with $\nabla g(p) \neq 0$ and $\phi$ be the trajectory of $\nabla g$ through $p$ using $\widehat{\nabla g(p)}$, whose existence is guaranteed by Lemma~\ref{lem:uniquephi}. Let $K = \set{x \in D}{g(x) \geq g(p)}$. Lemma~\ref{lem:compactset} implies $K$ is compact. Let $\{t_j\} \subset \RR_+$ be a sequence with $\lim_{j \to \infty} t_j = \infty$. Let $\left\{\witi{t}_j\right\}$ denote the tail of $\{t_j\}$ so that $\left\{\phi\left(\witi{t}_j\right)\right\} \subseteq K$ for all $j$. The sequence $\left\{\phi\left(\witi{t}_j\right)\right\}$ is an infinite set of points in a compact set, so it has an accumulation point $q$. Since $\phi$ is continuous (Definition~\ref{def:trajectory}), then $\dest(\phi) = \lim_{t\to \infty} \phi(t) = q$. 

First, we show $q$ is a critical point of $g$. It suffices to show $\nabla g\bigl(\phi(t)\bigr) \to 0$ as $t \to \infty$. Differentiating $\phi'(t)$ we find
\begin{equation}\label{eq:ddphi}
\begin{split}
\phi''(t) &=  \left(\nabla_{\frac{\partial}{\partial \phi}} \nabla g\bigl(\phi(t)\bigr)\right) \phi'(t)\\
&= \left(\nabla_{\frac{\partial}{\partial \phi}} \nabla g\bigl(\phi(t)\bigr)\right) \nabla g\bigl(\phi(t)\bigr)
\end{split}
\end{equation}
holds for all $t > 0$. The first and second derivatives of $g$ are bounded (property five of Definition~\ref{def:routing}), hence we may deduce from \eqref{eq:ddphi} that $\phi'$ is uniformly Lipschitz continuous for $t > 0$.

Since $g$ is continuous (Definition~\ref{def:routing}), $g_\infty := \lim_{t \to \infty} g\bigl(\phi(t)\bigr)$ exists and since $g$ is bounded (Lemma~\ref{lem:gbounded}), $g_\infty < \infty$. For $0 < t < \infty$
\[
g_\infty \geq g\bigl(\phi(t)\bigr) > g(p),
\]
so from \eqref{eq:dRHS}
\begin{equation}\label{eq:intineq}
\int_0^\infty \lVert \phi'(t) \rVert^2 \, dt = \int_{0}^\infty \frac{d}{dt} g \bigl(\phi(t)\bigr) \, dt = g_\infty - g(p) < \infty.
\end{equation}
Since $\phi'$ is uniformly Lipschitz continuous, \eqref{eq:intineq} implies
\[
\lim_{t \to \infty} \nabla g\bigl(\phi(t)\bigr) = \lim_{t \to \infty} \phi'(t) = 0
\]
as desired.


Finally, we show $q$ is a routing point in $D$. We find $g(q) > g(p) > 0$ because $g$ increases along $\phi$ as $t \to \infty$ (Lemma~\ref{lem:trajincr}). Hence, $q \in D$ is a routing point.
\end{proof}

We now show that the connected components of $\{g \neq 0\}$ can be decomposed in to a  disjoint union of stable manifolds.

\begin{lemma} \label{lem:disjunmorse} 
The component $D$ is a disjoint union of stable manifolds corresponding to the routing points contained in $D$; that is, 
\[
D = \coprod_{p \in R_D} W^s(p).
\]
where $R_D$ is the set of routing points of $g$ in $D$. 
\end{lemma}

\begin{proof} 
Let $R_D$ be the set of routing points of $g$ in $D$. Let $q \in D$ be arbitrary. Certainly $q \in W^s(q)$, so we may assume $\nabla g(q) \neq 0$. Let $\phi$ denote the trajectory of $\nabla g$ through $q$ using $\widehat{\nabla g(q)}$, whose existence is guaranteed by Lemma~\ref{lem:uniquephi}. It follows from Lemma~\ref{lem:origdestD} that there exists a routing point $r \in R_D$ such that $\dest(\phi) = r$. Hence $q \in W^s(r)$. This shows $D$ is a union of stable manifolds. It is a disjoint union due to the uniqueness of $\phi$ (Lemma~\ref{lem:uniquephi}).
\end{proof}

Now that we have a decomposition, the next natural question to ask is whether we can determine the dimension of each stable manifold. The definition of a stable manifold relies on a critical point, so one may believe that the dimension relies on the index of the critical point. To see this, we use the Stable Manifold Theorem, a fundamental result in the field of dynamical systems.

\begin{lemma} \label{lem:stabmanidim} 
If $p \in D$ is a routing point of $g$ with index $k$, then $W^s(p)$ is a smooth $k$-dimensional manifold.  
\end{lemma}

\begin{proof}
Let $p$ be a routing point of index $k$ of $g$ contained in $D$. The result in \cite[Theorem 4.2, Section 4.1, pp. 94]{Banyaga2004} has the same conclusion but the assumptions are that $g$ is a Morse function defined on a finite dimensional compact smooth Riemannian manifold. The function $g$ restricted to $D$ is Morse because $g$ is a routing function. The connected component $D$ of $\{g \neq 0\}$ is a finite dimensional smooth Riemannian manifold, but it is not compact. The compactness assumption is used in several spots throughout the proof of the cited theorem.
\begin{glists}{1em}{1.3em}{0.4em}{0em}
\itemr{(1)} There exist finitely many critical points of $g$ on the given manifold \cite[Corollary 3.3, Section 3.1, pp. 47]{Banyaga2004}.
\itemr{(2)} The gradient vector field $\nabla g$ generates a unique 1-parameter group of diffeomorphisms defined on $\RR \times D$ \cite[Section 4.1, pp. 94]{Banyaga2004}. 
\itemr{(3)} The destination of a trajectory is a critical point \cite[Corollary 3.19, Section 3.2, pp. 59]{Banyaga2004}.
\end{glists}
All of these issues can be addressed though.
\begin{glists}{1em}{1.3em}{0.4em}{0em}
\itemr{(1)} The manifold $D$ contains finitely many routing points because $g$ is a routing function. 
\itemr{(2)} This follows from the fact that the gradient vector field $\nabla g$ is complete (Remark~\ref{rem:uniquephi}).
\itemr{(3)} This is exactly Lemma~\ref{lem:origdestD}.
\end{glists}
\end{proof}

We expect all the routing points in a connected component to be connected via steepest ascent paths, so we expect each component to have a ``peak'' to ascend to; that is, we expect each component to have a local maximum. The simple observation follows from the routing function properties.

\begin{lemma}\label{lem:maxinD} 
The component $D$ contains a routing point of $g$ having index $n$. 
\end{lemma}

\begin{proof} 
Take $x_0 \in D$. Then $g(x_0) > 0$. Let $K = \set{x \in D}{g(x) \geq g(x_0)}$. The set $K$ is compact (Lemma~\ref{lem:compactset}), hence $g$ has a maximum $z$ on $K$. The maximum must occur on the interior of $K$. 

If the interior is non-empty, then there exists an open ball $B$ around $z$ such that $g(z) \geq g(x)$ for all $x \in B$. Hence $z$ is a local maximum of $g$. Since $z$ is a critical point of $g$ where $g(z) > 0$, $z$ is a routing point of $g$. According to property four of Definition~\ref{def:routing}, $z$ is nondegenerate, so $(\Hess g)(z)$ is non-singular and has $n$ negative eigenvalues since $z$ is a local maximum. Hence $z$ is a routing point having index $n$ \cite[Section 1.3]{Banyaga2004}. 

If the interior of $K$ is empty, choose $x_1 \in D$ such that $g(x_1) < g(x_0)$, which is possible due to the second property Definition~\ref{def:routing}. Let $\witi{K} = \set{x \in D}{g(x) \geq g(x_1)}$. Again, the set $\witi{K}$ is compact so $g$ has a maximum $\witi{z}$ on $K$. The interior of $\witi{K}$ is non-empty, so as argued before, $\witi{z}$ is a local maximum of $g$; that is $\witi{z}$ is a routing point having index $n$.
\end{proof}

Throughout this section we will use the notation $\partial W$ to denote the boundary of a stable manifold $W$.

\begin{lemma} \label{lem:nominminmaxmax} 
If $p$ is a routing point of $g$ of index $n$, then $\partial W^s(p)$ contains no routing points of index $n$. 
\end{lemma}

\begin{proof} 
Let $p$ be a routing point of $g$ of index $n$. Assume for a contradiction that $\partial W^s(p)$ contains a routing point $q$ of index $n$. Hence $q$ is a local maximum of $g$. Any neighborhood $U$ of $q$ must contain a point $y \in W^s(p)$ where $g(y) > g(q)$, contradicting the fact that $q$ is a local maximum. Hence, $\partial W^s(p)$ contains no routing points of index $n$.  
\end{proof}

\begin{lemma}\label{lem:stayinboundary} 
Let $r \in D$ be a routing point of $g$ of index $n$. Let $p \in \partial W^s(r) \cap D$ with $\nabla g(p) \neq 0$ and $\phi$ be the trajectory of $\nabla g$ through $p$ using $\widehat{\nabla g(p)}$. Then there exists a routing point $q \in \partial W^s(r) \cap D$ such that $\dest(\phi) = q$. 
\end{lemma}

\begin{proof} 
Let $r \in D$ be a routing point of $g$ of index $n$. Let $p \in \partial W^s(r) \cap D$ with $\nabla g(p) \neq 0$ and $\phi$ be the trajectory of $\nabla g$ through $p$ using $\widehat{\nabla g(p)}$, whose existence is guaranteed by Lemma~\ref{lem:uniquephi}. According to Lemma~\ref{lem:origdestD}, there exists a routing point $q \in D$ such that $\dest(\phi) = q$. Hence $p \in W^s(q)$. In fact, all the points along $\SA\left(g,p ,\widehat{\nabla g(p)}\right)$ (see Definition~\ref{def:trajectory}) are in $W^s(q)$. Since $\phi$ is continuous and $D$ is a disjoint union of stable manifolds (Lemma~\ref{lem:disjunmorse}), we find that $q \in \partial W^s(r)$. Thus $q \in \partial W^s(r) \cap D$ as desired.
\end{proof}

\begin{lemma}\label{lem:routstayinboundary} 
Let $r \in D$ be a routing point of $g$ of index $n$. Let $p \in \partial W^s(r) \cap D$ be a routing point of $g$ of index strictly less than $n$. If $v$ is an outgoing eigenvector of $(\Hess g)(p)$ tangent to $\partial W^s(r)$, then there exists a routing point $q \in \partial W^s(r) \cap D$ that is reachable from $p$ using $v$.  
\end{lemma}

\begin{proof} Let $r \in D$ be a routing point of $g$ of index $n$. Let $p \in \partial W^s(r) \cap D$ be a routing point of $g$ of index strictly less than $n$. Let $v$ be a outgoing eigenvector of $(\Hess g)(p)$ tangent to $\partial W^s(r)$. 

For $\varepsilon > 0$, let $p_\varepsilon = p + \varepsilon v$ and $\phi_\varepsilon$ denote the trajectory of $\nabla g$ through $p_\varepsilon$ using $\widehat{\nabla g(p_\varepsilon)}$. For sufficiently small $\varepsilon$, $\nabla g(p_\varepsilon) \neq 0$ so $q_\varepsilon = \dest(\phi_\varepsilon)$ exists and is a routing point of $g$ according to Lemma~\ref{lem:origdestD}. Let $\phi = \lim_{\varepsilon \to 0} \phi_\varepsilon$, then $\phi$ is the trajectory of $\nabla g$ through $\lim_{\varepsilon \to 0} p_\varepsilon = p$ using $v$ and its image is $\SA(g, p, v)$. This limit exists because $g$ is continuous and the trajectories are unique (Lemma~\ref{lem:uniquephi}). Furthermore, $\dest(\phi)$ exists and is a routing point $q = \lim_{\varepsilon \to 0} q_\varepsilon$ of $g$.

As argued in the proof of Lemma~\ref{lem:stabmanidim}, we may use the conclusions of the Stable Manifold Theorem \cite[Theorem 4.2, Section 4.1, pp. 94]{Banyaga2004}. This theorem says that the positive eigenvector $v$ lies in the span of the negative eigenvectors of $(\Hess g)(p)$. As a consequence, $\SA(g, p, v)$ must lie in $\partial W^s(r)$. Furthermore, $q$ must also lie in $\partial W^s(r)$. Certainly $q \in D$ as $\SA(g, p, v)$ is a steepest ascent path that cannot leave the component $D$. Thus $q$ is reachable from $p$ using $g$ and $v$ and $q \in \partial W^s(r) \cap D$ as desired.
\end{proof}

\begin{lemma}\label{lem:boundaryconnect2max} 
Let $r \in D$ be a routing point of $g$ of index $n$. Let $q$ be a routing point of $g$ on $\partial W^s(r) \cap D$. Then $q$ is connected to $r$ by steepest ascent paths using outgoing eigenvectors of $g$. 
\end{lemma}

\begin{proof} 
Let $r \in D$ be a routing point of $g$ of index $n$. Let $q$ be a routing point of $g$ on $\partial W^s(r) \cap D$. According to Lemma~\ref{lem:nominminmaxmax}, $q$ must be a routing point of index strictly less than $n$. Hence, $(\Hess g)(q)$ has at least one outgoing eigenvector, call it $v$. Recall from Lemma~\ref{lem:stabmanidim} that $W^s(r)$ is a smooth $n$-dimensional manifold.

If $v$ is not tangent to $\partial W^s(r)$, then $\SA(g, q, v)$ or $\SA(g, q, -v)$ lies in the stable manifold $W^s(r)$ because $D$ is a disjoint union of stable manifolds (Lemma~\ref{lem:disjunmorse}). Hence $r$ is reachable from $q$ using $v$ (or $-v$). We see $q$ is connected to $r$ by steepest ascent paths using outgoing eigenvectors of $g$.

If $v$ is tangent to $\partial W^s(r)$, according to Lemma~\ref{lem:routstayinboundary} there exists another routing point $q_2$ that is reachable from $q = q_1$ using $v$. The routing point $q_2$ has index strictly less than $n$, so as before, there exists a routing point $q_3$ that is reachable from $q_2$ using $v$. We repeat this process. The function $g$ is bounded (Lemma~\ref{lem:gbounded}) and there are finitely many routing points, so eventually the process will terminate, and we will find a routing point $q_k$, $k \geq 1$, where $(\Hess g)(q_k)$ has an outgoing eigenvector $v_k$ that is not tangent to $\partial W^s(r)$. The point $r$ is reachable from $q_k$ using $g$ and $v_k$ as before. We have found a sequence of routing points $q_1 \dotsc, q_k$, $k \geq 2$ such that $q_i$ is reachable from $q_{i-1}$ using $g$ and an outgoing eigenvector of $(\Hess g)(q_{i-1})$. Thus the point $q$ is connected to $r$ by steepest ascent paths using outgoing eigenvectors of $g$ by the connectivity path $q_1, \dotsc, q_k, r$ and the corresponding trajectories connecting the routing points $q_1, \dotsc, q_k, r$.
\end{proof}

\begin{definition}
Let $p, q \in D$, $p \neq q$, be routing points of $g$ of index $n$. We say $W^s(p)$ is \textit{adjacent} to $W^s(q)$ if $D \cap \partial W^s(p) \cap \partial W^s(q)$ is non-empty.
\end{definition}

\begin{lemma}\label{lem:intersectmaxstable} 
Let $p, q \in D$, $p \neq q$, be routing points of $g$ of index $n$. If $W^s(p)$ is adjacent to $W^s(q)$, then $D \cap \partial W^s(p) \cap \partial W^s(q)$ must contain a routing point of $g$.  
\end{lemma}

\begin{proof} 
Let $p, q \in D$, $p \neq q$, be routing points of $g$ of index $n$. Assume $Z = D \cap \partial W^s(p) \cap \partial W^s(q)$ is non-empty. Suppose $Z$ does not contain a routing point of $g$. As $Z$ is non-empty, there exists a point $x \in Z$ that is not a routing point of $g$. According to Lemma~\ref{lem:stayinboundary}, there exists a routing point $q \in Z$. However, this contradicts our assumption. Hence, $Z$ contains a routing point of $g$.
\end{proof}

We now prove Theorem~\ref{thm:connresult}.

\begin{proof}[Proof of Theorem~\ref{thm:connresult}] 
Let $R$ denote the set of routing points of $g$ in $D$ and $p, q \in D$, $p \neq q$ be arbitrary. We will show $p$ and $q$ are connected by steepest ascent paths using outgoing eigenvectors of $g$. We may assume without loss of generality that $p$ and $q$ are routing points of $g$, otherwise we can always ascend to one using Lemma~\ref{lem:origdestD} if $\nabla g(p) \neq 0$ or $\nabla g(q) \neq 0$. Let $m_1, \dotsc, m_\ell$ denote the routing points in $R$ having index $n$. We see $\ell \geq 1$ due to Lemma~\ref{lem:maxinD}. We see that $\lvert R \rvert > 1$ because $p$ and $q$ are both distinct routing points of $g$.

Suppose first that $\ell = 1$. According to Lemma~\ref{lem:stabmanidim}, $W^s(m_1)$ is $n$-dimensional and the stable manifolds for the points in $R \setminus \{m_1\}$ have dimension strictly less than $n$. As $D$ is a disjoint union of stable manifolds of the routing points in $R$ (Lemma~\ref{lem:disjunmorse}), it follows that the points in $R \setminus \{m_1\}$ lie on $\partial W^s(m_1)$. We see for all $r \in R \setminus \{m_1\}$, $r$ is connected to $m_1$ by steepest ascent paths using outgoing eigenvectors (Lemma~\ref{lem:boundaryconnect2max}), hence any two routing points in $D$ can be connected using steepest ascent paths using outgoing eigenvectors of $g$. 

Now suppose $\ell > 1$. According to Lemma~\ref{lem:stabmanidim}, for all $i$, $W^s(m_i)$ is $n$-dimensional and the stable manifolds for the points in $R \setminus \{m_1, \dotsc, m_\ell\}$ have dimension strictly less than $n$. If $p$ (or $q$) is a routing point with index strictly less than $n$, then it must lie on the boundary of some stable manifold $W^s(m_i)$. According to Lemma~\ref{lem:boundaryconnect2max}, we can connect $p$ (or $q$) to $m_i$ by steepest ascent paths using outgoing eigenvectors. Hence, we may assume without loss of generality that $p$ and $q$ are routing points having index $n$. We will connect $p$ and $q$ by looking at a sequence of adjacent stable manifolds of dimension $n$ as seen in Figure~\ref{fig:adjmanifolds}. 

\begin{figure}[ht]
\centering
\subfloat[]{\includegraphics{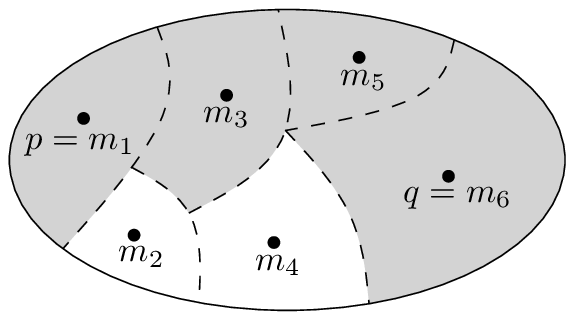}\label{fig:adjmanifolds}}
\qquad 
\subfloat[]{\includegraphics{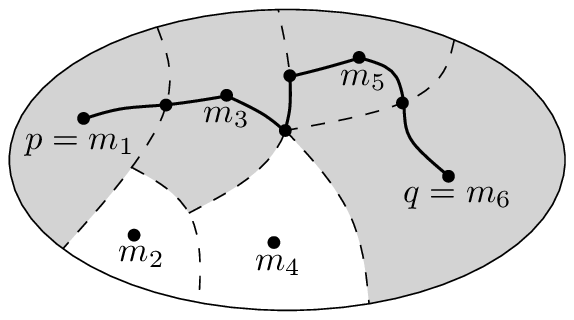}\label{fig:connadjmanifolds}}
\caption{The dotted lines are the boundaries of $W^s(m_i)$.}
\end{figure} 

It suffices to show that we can connect any two $m_i$, $m_j$ whose stable manifolds $W^s(m_i)$ and $W^s(m_j)$ are adjacent because $D$ is a disjoint union of stable manifolds of the routing points in $R$ (Lemma~\ref{lem:disjunmorse}). From Lemma~\ref{lem:intersectmaxstable}, we know two adjacent manifolds have a routing point in common in their boundary.  According to Lemma~\ref{lem:boundaryconnect2max}, we can connect this common routing point to both $m_i$ and $m_j$ by steepest ascent paths using outgoing eigenvectors. Hence we can connect $m_i$ and $m_j$ by steepest ascent paths using outgoing eigenvectors. We illustrate this in Figure~\ref{fig:connadjmanifolds}. This completes the proof of Theorem~\ref{thm:connresult}.
\end{proof}

We now prove Theorem~\ref{thm:correct}.

\begin{proof}[Proof of Theorem~\ref{thm:correct}] 
Let $f$, $p$, $q$ be the inputs to \textbf{Connectivity} satisfying the specification. Suppose \textbf{Connectivity} terminated with output $t$. Let 
\[
g = \frac{f^2}{U^\gamma} \text{ where } U = (x_1 - c_1)^2 + \dotsb + (x_n - c_n)^2 + 1,  \gamma = \deg(f)+1
\]
be the function formed in step 2. First, we claim that the set $R$ formed in step 3 is the set of routing points of $g$. We observe that
\begin{equation}\label{eq:dg}
\nabla g = \frac{f}{U^{\gamma+1}}(2 \nabla f U - \gamma f \nabla U )
\end{equation}
so
\begin{align*}
R &= \set{x \in \RR^n}{\nabla g(x) = 0, g(x) \neq 0} \\
&= \set{x \in \RR^n}{2 \nabla f(x) U(x) - \gamma f(x) \nabla U(x) = 0, f(x) \neq 0}.
\end{align*}
Let $\mathcal{F} = 2 \nabla f U - \gamma f \nabla U$. We see that the
$V_{\new{\mathbb{R}}}(\mathcal{F})$ contains exactly the routing
points of $g$ and the singular points of $f$ because $U$ is
non-zero. In step 3, we remove the finitely many singular points of
$f$ from $V_{\new{\mathbb{R}}}(\mathcal{F})$, leaving us with the
correct set of routing points. The set of routing points is finite
because $V_{\new{\mathbb{C}}}(\mathcal{F})$ is zero-dimensional.

We now claim that $g$ is a routing function. The function $g$ is $C^2$
because it is a rational function where the denominator is
positive. According to step 2, the finitely many routing points of $g$
are all nondegenerate because $\det (\Hess g)(r) \neq 0$ for all
$r \in R$. The choice of $\gamma = \deg(f) + 1$ guarantees the
property that $g$ vanishes at infinity (property two) because the
degree of the numerator is smaller than the degree of the
denominator. Certainly the function $g$ is nonnegative. To understand
why the first derivative of $g$ is bounded, we observe in
\eqref{eq:dg} that each component of $\nabla g$ is a rational function
where the degree of the numerator is smaller than the degree of the
denominator, which is nonnegative. A similar argument holds for each
component of $\Hess g$. Hence $g$ satisfies the properties in the
definition of a routing function.

Observe that $g = 0$ if and only if $f = 0$. Due to Theorem~\ref{thm:connresult}, we know the routing points of $g$ on a connected component of $\{f \neq 0\}$ are connected by steepest ascent paths using outgoing eigenvectors of $g$. It is important to observe that these steepest ascent paths do not cross $f = 0$ due to Lemma~\ref{lem:trajincr}. In steps 5 and 6, we use the certified \textbf{Destination} algorithm to determine which routing points are adjacent to one another via steepest ascent paths using outgoing eigenvectors. The matrix $A$ is the adjacency matrix for the graph whose vertices are the routing points and whose edges are the steepest ascent paths connecting them. Hence, the matrix $M$, the reflexive, symmetric, transitive closure of $A$, satisfies the condition that $M_{ij} = 1$ if and only if $r_i, r_j \in R$ lie in a same connected component of $\{f \neq 0\}$.

We claim that the point $p$ can be connected to a routing point $r_i$ lying in the same connected component of $\{f \neq 0\}$. If $\nabla g(p) = 0$ then $p$ is a routing point of $g$ because $f(p) > 0$ implies $g(p) > 0$; that is, there exists $i$ such that $r_i = p$. Otherwise, if $\nabla g(p) \neq 0$, let $\phi_p$ be the trajectory of $\nabla g$ through $p$ using $\widehat{\nabla g(p)}$. According to Lemma \ref{lem:origdestD}, there exists $i$ such that the destination of $\phi_p$ is a routing point $r_i$. The index $i$ in this case can be determined using the \textbf{Destination} algorithm (step 7). A similar argument holds for $q$; the point $q$ can be connected to a routing point $r_j$ lying in the same connected component of $\{f \neq 0\}$, with this index being determined in step 8. We use the connectivity matrix $M$ in step 9 to determine if $r_i$ and $r_j$ lie in a same connected component of $\{f \neq 0\}$ to conclude whether $p$ and $q$ lie in a same connected component.
\end{proof}


\section{Termination} \label{sec:termination}

In this section, we will prove that the termination of the algorithm \textbf{Connectivity} in the form of Theorem~\ref{thm:terminates}. For this, we must show that the perturbation step completes after a finite number of iterations. We will show in Theorem~\ref{thm:zariski} that there is only a small (measure zero) set of parameters for which the function $g$ formed in \textbf{Connectivity} is not a routing function. Hence we are guaranteed to find a routing function by finitely many perturbation of these parameters on the integer grid. 

\begin{theorem}\label{thm:zariski} For all nonzero $f \in \RR[x_1, \dotsc, x_n]$ there exists a semi-algebraic set $S \subset \RR^{n}$ such that $\dim \left(\RR^n \setminus S\right) < n$ and for all $(c_1, \dotsc, c_n) \in S$ the mapping $g \colon \RR^n \to \RR$ defined by
\begin{equation}\label{eq:generalg}
g = \frac{f^2}{\bigl((x_1 - c_1)^2 + \dotsb + (x_n - c_n)^2 + 1\bigr)^{\deg(f)+1}}
\end{equation}
is a routing function.
\end{theorem}

Before we prove Theorem~\ref{thm:zariski}, we recall defintions from semi-algebraic geometry \cite{Basu2003}. Let $A \subset \RR^m$ and $B \subset \RR^n$ be two semi-algebraic sets. A function $f \colon A \to B$ is \textit{semi-algebraic} if its graph is a semi-algebraic subset of $\RR^{m+n}$. For open $A$, the set of semi-algebraic functions from $A$ to $B$ for which all partial derivatives up to order $\ell$ exist and are continuous is denoted $\sS^{\ell}(A,B)$. The class $\sS^\infty(A,B)$ is the intersection of $\sS^{\ell}(A,B)$ for all finite $\ell$. A \textit{$\sS^\infty$-diffeomorphism} $\phi$ from a semi-algebraic open $U \subset \RR^n$ to a semi-algebraic open $V \subset \RR^n$ is a bijection from $U$ to $V$ such that $\phi \in \sS^\infty(U,V)$ and $\phi^{-1} \in \sS^\infty(V,U)$. 

Let $\ell \geq 0$. A semi-algebraic $A \subset \RR^n$ is a \textit{$\sS^\infty$-submanifold of $\RR^n$ of dimension $\ell$} if for every $x \in A$ there exists a semi-algebraic open $U$ of $\RR^n$ and an $\sS^\infty$-diffeomorphism $\phi$ from $U$ to a semi-algebraic open neighborhood $V$ of $x$ in $\RR^n$ such that $\phi(0)=x$ and
\[
\phi\Bigl(U \cap \bigl(\RR^\ell \times \{0\}\bigr)\Bigr) = A \cap V,
\]
where $\RR^\ell \times \{0\} = \set{(a_1,\dotsc, a_\ell, 0, \dotsc, 0) \in \RR^n}{(a_1, \dotsc, a_\ell) \in \RR^\ell}$.

\begin{lemma}\label{lem:sardcor2}  Let $A$ be an open $\sS^\infty$ manifold and $f \in \sS^\infty(A, \RR^m)$. Then there exists a semi-algebraic set $S \subseteq \RR^m$ and semi-algebraic open set $U \subseteq A$ such that for all $y^0 \in S$, $\dim \set{x \in U}{f(x) - y^0 =0} = \dim A - m$. Furthermore, $\dim \left(\RR^m \setminus S\right) < m$.
\end{lemma}

\begin{proof} Let $A$ be an open $\sS^\infty$ manifold and $f \in \sS^\infty(A, \RR^m)$. By the semi-algebraic version of Sard's Theorem \cite[Theorem 5.56, Section 9, pp. 192]{Basu2003}, the set $C$ of critical values of $f$ is a semi-algebraic set in $\RR^m$ and $\dim \left(\RR^m \setminus S\right) < m$. Let $S = \RR^m \setminus C$ be its complement (which is a semi-algebraic set). For any $y^0 \in S$ there exists $x^0  \in A$ where $y^0 = f(x^0)$. Let $g \colon A \to \RR^m$ be defined by $g(x) = f(x) - y^0$. Since $y^0 \notin C$, $\rnk \D g(x^0) = m$ because $f$ has full rank on a neighborhood of $x^0$. By the Constant Rank Theorem \cite[Theorem 5.57, Section 9, pp. 192]{Basu2003} there exists a semi-algebraic open neighborhood $U$ of $x^0$ in $A$ where $\dim \set{x \in U}{f(x) - y^0 = 0} = \dim \ker g = \dim A - \rnk g = \dim A -m$.
\end{proof}

We now have the machinery to present the proof of Theorem~\ref{thm:zariski}. 

\begin{proof}[Proof of Theorem~\ref{thm:zariski}] Assume $f \in \RR[x_1, \dotsc, x_n]$ is non-zero. For notational purposes let $x = (x_1, \dotsc, x_n)$. We will find a set $S$ so that $g$ is a routing function in the following manner. First, let $p = (p_1, \dotsc, p_n)$ be the mapping where $p_i \colon A \subset \RR^{n+1} \to \RR$ is defined by
\begin{equation}\label{eq:p}
p_i(x, t) = -\partial_i f(x) t + x_i
\end{equation}
and $A = \set{(x,t) \in \RR^n \times \RR}{t \neq 0 \text{ and } f(x) \neq 0}$. Observe that $A$ is an open $\sS^\infty$ manifold of dimension $n + 1$ and $p \in \sS^\infty(A, \RR^n)$. By Lemma~\ref{lem:sardcor2} there exists a semi-algebraic set $S_1 \subseteq \RR^n$ and semi-algebraic open set $U_1 \subseteq A \subseteq \RR^n \times \RR$ such that for all $y \in S_1$, $\dim V_1 = \dim A - n = (n+1) - n = 1$ where $V_1 = \set{(x,t)\in U_1}{p(x,t) - y = 0}$.

Let $y = (y_1, \dotsc, y_n) \in S_1$. Define $q \colon B \subset \RR^{n+1} \to \RR$ to be
\begin{equation}\label{eq:q}
q(x, t) = \frac{(x_1 - y_1)^2 + \dotsb + (x_n - y_n)^2 + 1}{tf(x)}
\end{equation}
where $B = A \cap V_1$. Observe $B$ is an open $\sS^\infty$ manifold of dimension $1$ and $q \in \sS^\infty(B, \RR)$. From Lemma~\ref{lem:sardcor2} we find a semi-algebraic set $S_{2,y} \subseteq \RR$ and semi-algebraic open set $U_2 \subseteq B \subseteq \RR^n \times \RR$ such that for all $\tilde{y} \in S_{2,y}$, $\dim V_{2,y} = \dim B - 1 = 1- 1 = 0$ where $V_{2,y} = \set{(x,t)\in U_2}{q(x,t) - \tilde{y} = 0}$. 

We claim $S_{2,y} = \RR$. From  Lemma~\ref{lem:sardcor2} we know 
\[
\RR \setminus S_{2,y} = \{\text{critical values of }q\}. 
\]
For notational purposes let
\[
W(x) = (x_1 - y_1)^2 + \dotsb + (x_n - y_n)^2 + 1.
\]
so $q(x, t) = \frac{W(x)}{tf(x)}$. Consider the system $\nabla q(x, t) = 0$:
\begin{align*}
\begin{bmatrix}\partial_{x_1} q(x, t) \\ \vdots \\ \partial_{x_n} q(x, t)  \\ -\frac{W(x)}{f(x)t^2}\end{bmatrix} = \begin{bmatrix}0 \\ \vdots \\ 0 \\ 0\end{bmatrix}.
\end{align*}
For all $(x,t) \in B$, we have $f(x) \neq 0$, $t \neq 0$, and $W(x) \neq 0$, which leads us to conclude
\[
-\frac{W(x)}{f(x)t^2} = 0
\]
is not true. Hence, the mapping $q$ has no critical points. Since the set of critical values of $q$ is empty, $S_{2,y} = \RR$.

Let $S = S_1$. Clearly $S \subset \RR^{n}$ and $S$ is semi-algebraic. The fact $\dim \left(\RR^n \setminus S\right) < n$ follows directly from Lemma~\ref{lem:sardcor2}.

Let $c = (c_1, \dotsc, c_n) \in S$, $\gamma \in S_{2,c} \setminus \{ 0\} = \RR \setminus \{0\}$, 
\[
U(x) = (x_1 - c_1)^2 + \dotsb + (x_n - c_n)^2 + 1
\]
and
\[
g(x) = \frac{f(x)^2}{U(x)^\gamma}.
\]
Let $R = \set{x \in \RR^n}{\nabla g(x) = 0 \text{ and } f(x) \neq 0}$ denote the set of routing points of $g$. We claim $R$ is finite. Observe
\begin{align}
\nabla g(x) &= \frac{2f(x) \nabla f(x)U(x)^\gamma - \gamma f(x)^2 U(x)^{\gamma-1} \nabla U(x)}{U(x)^{2\gamma}} \\
&= \frac{f(x)U(x)^{\gamma-1}\bigl[2\nabla f(x)U(x) - \gamma f(x) \nabla U(x)\bigr]}{U(x)^{2\gamma}} \\
&= \frac{f(x)}{U^{\gamma+1}}\bigl[2\nabla f(x)U(x) - \gamma f(x)\nabla U(x)\bigr]. \label{eq:gprime}
\end{align}
Let
\begin{align*}
P(x) &= \frac{f(x)}{U(x)^{\gamma+1}}\\
Q(x) &= 2\nabla f(x)U(x) - \gamma f(x)\nabla U(x).
\end{align*}
so $\nabla g(x) = P(x)Q(x)$. For all $x$, $P(x) \neq 0$, so $x \in R$ if and only if $Q(x) = 0$ and $f(x) \neq 0$. Let us rewrite $Q(x) = 0$ in the following way:
\begin{align*}
0 &= 2\nabla f(x)U(x) - \gamma f(x) \nabla U(x) \\
\begin{bmatrix}0\\\vdots\\0\end{bmatrix}&= 
\begin{bmatrix}
2\partial_{x_1} f(x)U(x) - 2\gamma f(x) (x_1-c_1)\\
\vdots\\
2\partial_{x_n} f(x)U(x) - 2\gamma f(x) (x_n-c_n)
\end{bmatrix} \\
\begin{bmatrix}c_1\\\vdots\\c_n\end{bmatrix}&= 
\begin{bmatrix}
-\partial_{x_1} f(x)\frac{U(x)}{\gamma f(x)} + x_1\\
\vdots\\
-\partial_{x_n} f(x)\frac{U(x)}{\gamma f(x)} + x_n
\end{bmatrix}. 
\end{align*}
Let $t = \frac{U(x)}{\gamma f(x)}$ so
\begin{align}
c_1 &= -\partial_{x_1} f(x)t + x_1 \notag \\
&\;\vdots \label{eq:newsystem} \\
c_n &=  -\partial_{x_n} f(x)t + x_n \notag\\
\gamma & = \frac{U(x)}{tf(x)} \notag
\end{align}
Since $\gamma \neq 0$, $x \in R$ if and only if $x$ satisfies \eqref{eq:newsystem} and $f(x) \neq 0$. Using our previous notation, rewrite \eqref{eq:newsystem} as 
\begin{align}
0 &= p_1(x,t) - c_1\notag \\
&\;\vdots \label{eq:newsystem2} \\
0 &=  p_n(x,t) - c_n\notag\\
0 & = q(x,t) - \gamma. \notag
\end{align}
Thus, when $(c_1, \dotsc, c_n) \in S$ and $\gamma \neq 0$, $x \in R$ if and only if $x$ satisfies \eqref{eq:newsystem2} and $f(x) \neq 0$. Suppose now that $x \in R$. It follows that $t = \frac{U(x)}{\gamma f(x)} \neq 0$ and $q(x_1, \dotsc, x_n, t) - \gamma = 0$, implying $(x_1, \dotsc, x_n, t) \in V_{2,c}$. As shown earlier, $\dim V_{2,c} = 0$. Combining this with the fact that $R \times (t \neq 0) \subset V_{2,c}$ implies $\dim R = 0$. The set $R$ is finite because $R$ is semi-algebraic and has dimension zero.
 
We now show the routing points of $g$ are nondegenerate. From \eqref{eq:gprime} we see
\[
(\Hess g)(x) = JP(x)Q(x) + P(x) JQ(x)
\]
where $JP$ is the jacobian of $P$. When we evaluate $\Hess g$ at a point $x \in R$,
\[
(\Hess g)(x) = JP(x)Q(x) + P(x) JQ(x) = P(x) JQ(x).
\]
Hence
\[
\det (\Hess g)(x) = \det \bigl(P(x) JQ(x)\bigr) = P(x)^n \det JQ(x).
\]
Clearly $P(x) \neq 0$. When $(c_1, \dotsc, c_n) \in S$ then $(c_1, \dotsc, c_n)$ is not a critical value of $p$. Also $\gamma$ is not a critical value of $q$. Thus $\det JQ(x) \neq 0$. It follows $\det (\Hess g)(x) \neq 0$ as desired.

What we have shown so far is that if $(c_1, \dotsc, c_n) \in S$ and $\gamma \neq 0$, then the function
\[
g = \frac{f^2}{\bigl((x_1 - c_1)^2 + \dotsb + (x_n - c_n)^2 + 1\bigr)^{\gamma}}
\]
has finitely many routing points that are all nondegenerate. The choice of $\gamma = \deg(f) + 1$ guarantees the function $g$ vanishes at infinity (property two) because the degree of the numerator is smaller than the degree of the denominator. Certainly the function $g$ is nonnegative. To understand why the first derivative of $g$ is bounded, we observe in \eqref{eq:dg} that each component of $\nabla g$ is a rational function where the degree of the numerator is smaller than the degree of the denominator, which is nonnegative. A similar argument holds for each component of $\Hess g$. Hence the function $g$ is a routing function, as desired.
\end{proof} 

\begin{theorem} Algorithm \textbf{Connectivity} terminates. \label{thm:terminates}
\end{theorem}

\begin{proof} Let $f$, $p$, $q$ be the inputs to \textbf{Connectivity} satisfying the specification. To show Algorithm \textbf{Connectivity} terminates, first we must show that the loop in step 2 terminates in a finite number of iterations. Let $S$ be the semi-algebraic set from Theorem~\ref{thm:zariski} for the given $f$. According to Theorem~\ref{thm:zariski} the set of choices for $(c_1, \dotsc, c_n)$ for which 
\[
g = \frac{f^2}{\bigl((x_1 - c_1)^2 + \dotsb + (x_n - c_n)^2 + 1\bigr)^{\deg(f)+1}}
\]
is not a routing function is ``small'' since $\dim (\RR^n \setminus S) < n$. Hence, after a finite number of perturbations on the integer grid, we are guaranteed to find a parameter $(c_1, \dotsc, c_n) \in S$ which will guarantee $g$ is a routing function. 

Let $(c_1, \dotsc, c_n) \in S$ and 
\begin{align*}
\gamma &= \deg(f) + 1,\\ 
U &= (x_1 - c_1)^2 + \dotsb + (x_n - c_n)^2 + 1,\\
V_{\new{\mathbb{C}}}(\mathcal{F}) &= \bigl\{2 \cdot (\partial_{x_i} f) \cdot U - \gamma \cdot f \cdot (\partial_{x_i}U)\bigr\}_{i=1}^n.
\end{align*} 
\new{Using e.g. Gr\"obner bases, it can be established computationally
  that $V_{\new{\mathbb{C}}}(\mathcal{F})$ is zero-dimensional}. 

We see that the loop terminates because each of the finitely many
routing points of $g$, the set of points
$r \in V_{\new{\mathbb{C}}}(\mathcal{F})$ where $f(r) \neq 0$, are
nondegenerate.

The rest of the algorithm terminates because there are finitely many
routing points, the Hessian at each of these routing points has
finitely many outgoing eigenvectors, and the algorithm
\textbf{Destination} terminates.
\end{proof}


\section{Examples}\label{examples}

In this section, we present several non-trivial examples using input polynomials in two and three variables. Each example will illustrate the routing points and steepest ascent paths connecting them. The algorithms were implemented in the Maple computer algebra system. 

\begin{figure}[p]
\centering
\subfloat[]{\includegraphics[width=3in]{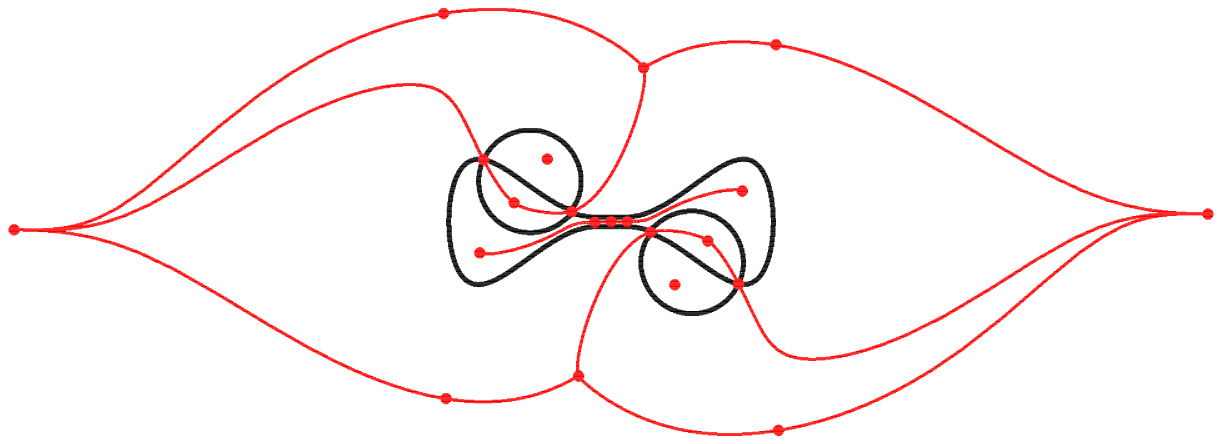}\label{fig:ex1cp}}
\subfloat[]{\includegraphics{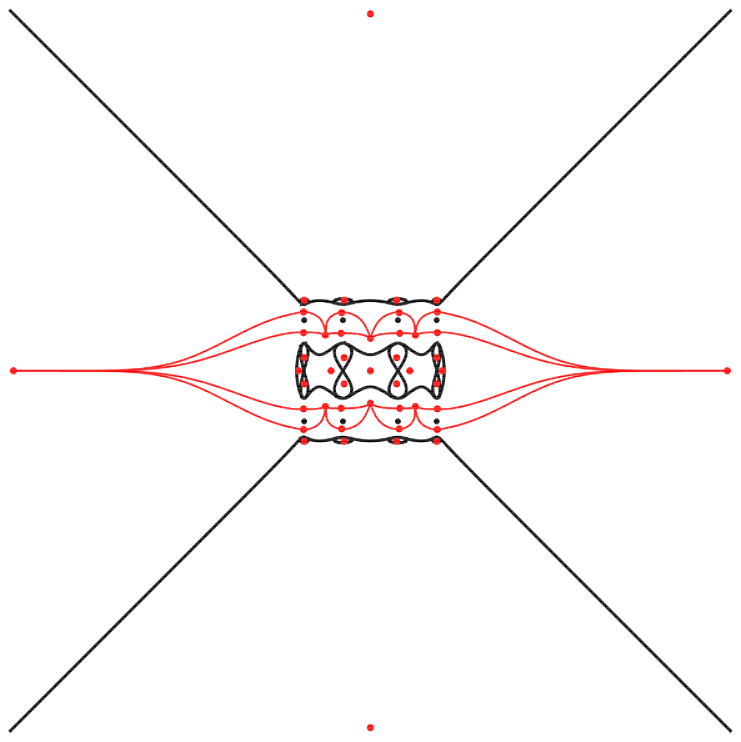}\label{fig:ex2cp}}

\bigskip

\bigskip

\bigskip

\subfloat[]{\includegraphics{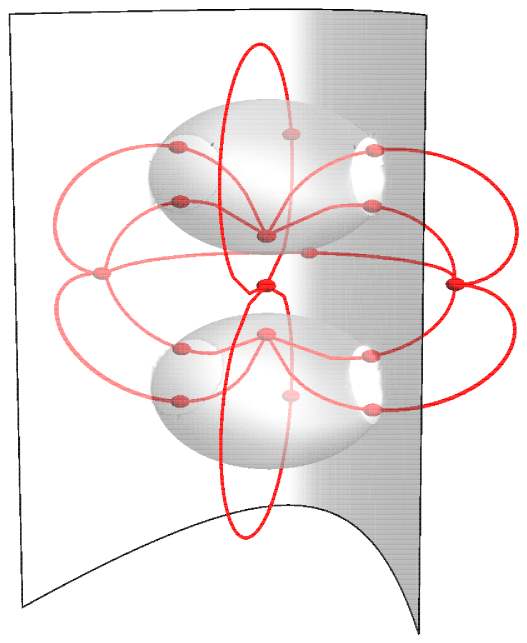}\label{fig:ex3cp}}
\subfloat[]{\includegraphics{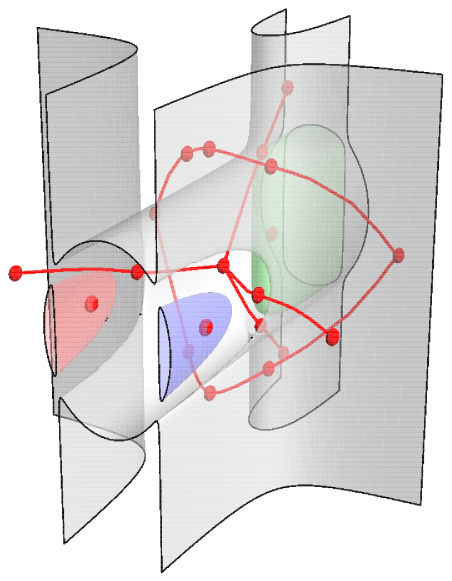}\label{fig:ex4cp}}
\end{figure}

\begin{example} \label{ex:cp1} Let
\begin{align*}
f &=1280000 x^{10}+2560000 x^8 y^2-2016000 x^8+1280000 x^7 y+1280000 x^6 y^4\\
& \quad -2336000 x^6 y^2 +793800 x^6-1280000 x^5 y-1280000 x^4 y^4+1056000 x^4 y^2\\
& \quad -59080 x^4+2560000 x^2 y^4-738560 x^2 y^2+736 x^2+1280000 x y^3-1280 x y\\
& \quad +1280000 y^6+222720 y^4+57576 y^2-45.
\end{align*}
In Figure~\ref{fig:ex1cp}, the curve $\{f \neq 0\}$ is shown in black while the routing point and steepest ascent paths are shown in red. The connectivity matrix formed had size $21 \times 21$.
\end{example}

In Example~\ref{ex:cp1} we see that the curve has many ``narrow'' gaps. The numeric methods for solving this problem would likely miss the narrow gaps, often producing wrong outputs. However, our algorithm presented in this article correctly catches all the narrow gaps.

\begin{example} Let
\begin{align*}
f &= 4096 x^{16}-16384 x^{14}+26624 x^{12}-22528 x^{10}-1024 x^8 y^4+1024 x^8 y^2 \\
& \quad + 10496 x^8 + 2048 x^6 y^4-2048 x^6 y^2-2560 x^6-1280 x^4 y^4+1280 x^4 y^2\\
& \quad +256 x^4 + 256 x^2 y^4 - 256 x^2 y^2-4096 y^{16}+16384 y^{14}-26624 y^{12}\\
& \quad +22528 y^{10}-10560 y^8+2688 y^6-352 y^4+32 y^2-1.
\end{align*}
In Figure~\ref{fig:ex2cp}, the curve $\{f \neq 0\}$ is shown in black while the routing point and steepest ascent paths are shown in red. The connectivity matrix formed had size $47 \times 47$.
\end{example}

\begin{example} Let
\begin{align*}
f &= -31 - 16 x^2 + 8 x^4 + 4 x^6 + 16 y + 16 x^2 y + 4 x^4 y - 32 y^2 + 
 8 x^4 y^2 + 16 y^3 + 8 x^2 y^3\\
& \quad - 8 y^4 + 4 x^2 y^4 + 4 y^5 + 
 96 z^2 - 64 x^2 z^2 + 8 x^4 z^2 - 48 y z^2 + 8 x^2 y z^2 - 
 16 y^2 z^2  \\
& \quad + 8 x^2 y^2 z^2 + 8 y^3 z^2 - 8 z^4 + 4 x^2 z^4 + 4 y z^4.
\end{align*}
In Figure~\ref{fig:ex3cp}, the semi-algebraic set $\{f = 0\}$ consists of one connected component which we show in light gray while the routing point and steepest ascent paths are shown in red. The connectivity matrix formed had size $16 \times 16$.
\end{example}

\begin{example} Let
\begin{align*}
f &= 20 x^4 y+20 x^2 y z^2-60 x^2 y+20 x^2-20 y z^2+40 y+20 z^2-41.
\end{align*}
In Figure~\eqref{fig:ex4cp}, the semi-algebraic set $\{f = 0\}$ consists of four connected components which we show in light gray, light red, light blue, and light green, respectively, while the routing point and steepest ascent paths are shown in red. The connectivity matrix formed had size $20 \times 20$.
\end{example}

\section{Conclusion and Future Works}\label{conclusion}

  In this paper we designed an algorithm for determining whether two points lie in a same connected component of a semi-algebraic set defined by a single polynomial inequation. We proved the method to be correct using modified results from Morse theory. Furthermore, we showed the method terminates using results from semi-algebraic geometry. We presented several non-trivial examples demonstrating the rigorousness of our method.

There are two problems that will be considered in the second forthcoming paper. First, we will give an algorithm for \textbf{Destination}. Care must be taken that the method to trace the steepest ascent paths using outgoing eigenvectors be rigorous. One such approach is to use interval based methods \cite{Makino2003, Moore2009, Nedialkov1999}. Researchers have used approaches like this in the past \cite{Vegter2012}, however their methods would need to be adapted carefully for our problem. 

The second problem to consider would be to find an upper bound on the length of the path connecting the two input points $p, q$ using the steepest ascent paths we discussed. Using techniques similar to those in \cite{DAcunto2004}, a bound can be derived in terms of the number of variables, the degree of the polynomial $f$, and the height of $f$. Such a result is interesting on its own, and will be used when determining the complexity of the method given in this article.

\begin{acknowledgment}
Hoon Hong's research was partially supported by the US National Science
Foundation (NSF 1319632). Mohab Safey El Din is supported by the joint 
ANR-FWF grant ANR-19-CE48-0015 ECARP, the ANR grants ANR-18-CE33-0011 
Sesame and ANR-19-CE40-0018 DeRerum Natura, the European Union’s Horizon 
2020 research and innovation program under the Marie Skłodowska-Curie 
Actions – Innovative Training Networks grant agreement N. 813211 POEMA.
\end{acknowledgment}

\bibliographystyle{elsarticle-harv}
\bibliography{connectivity,Hong,Rohal1,Rohal2,Rohal3,Rohal4}


\end{document}